\documentclass[10pt]{article}
\usepackage{booktabs}
\usepackage{amsthm}
\usepackage{bbm}
\usepackage[title,titletoc]{appendix}
\usepackage{lineno}
\usepackage{graphicx}
\usepackage{amsmath,amssymb,amsthm,amsfonts}
\usepackage{amssymb}
\usepackage{color}
\usepackage{amsfonts}
\usepackage{amssymb}
\usepackage{amsthm}
\usepackage{amsmath}
\usepackage{empheq}
\usepackage{indentfirst}
\usepackage{cite}
\usepackage{mathrsfs}
\usepackage{lineno}
\allowdisplaybreaks[4]
\usepackage{}
\usepackage{graphicx,graphics,subfigure,xcolor}
\usepackage{amsmath,amssymb,amsbsy,amsthm,amsfonts}
\usepackage{subfigure}
\usepackage{amssymb}
\usepackage{mathrsfs}
\usepackage{color}
\usepackage{amsfonts}
\usepackage{amssymb}
\usepackage{amsthm}
\usepackage{amsmath}
\usepackage{empheq}
\usepackage{indentfirst}
\usepackage{cite}
\usepackage{mathrsfs}
\usepackage{extarrows}
\usepackage{array}
\allowdisplaybreaks[4]
\newtheorem{thm}{Theorem}[section]
\newtheorem{cor}[thm]{Corollary}
\newtheorem{lem}[thm]{Lemma}\newtheorem{prop}[thm]{Proposition}

\theoremstyle{definition}

\theoremstyle{Remark}
\newtheorem{rem}{Remark}[section]

\numberwithin{equation}{section}

\newcommand{\beq}{\begin{equation}}
\newcommand{\eeq}{\end{equation}}
\DeclareMathOperator{\st}{s.t.}

\newcommand{\lam}{\lambda}
\newcommand{\alp}{\alpha}

\newcommand{\R}{\mathbb{R}}
\newcommand{\N}{\mathbb{N}}

\def\ni{\noindent}
\def\proof{{\ni\bf Proof:\quad}}
\def\proofend{{\hfill$\Box$}\\}
\def\qed{\hfill$\square$\vspace{6pt}}

\newcommand{\bsub}{\begin{subequations}}
\newcommand{\esub}{\end{subequations}$\!$}

\title{Study on the behaviors of rupture solutions for a class of elliptic MEMS equations in $\R^2$}

 \author{Qing Li\thanks{School of Mathematical Sciences, East China Normal University, Shanghai 200241,   P.R. China. Email: \texttt{51185500011@stu.ecnu.edu.cn}. },\,
 Yanyan Zhang\thanks{Corresponding author. School of Mathematical Sciences,  Key Laboratory of MEA(Ministry of Education) and Shanghai Key Laboratory of PMMP,  East China Normal University, Shanghai 200241, China. Email: \texttt{yyzhang@math.ecnu.edu.cn}. Y. Zhang is sponsored by
  NSFC [No.12271505] and
  STCSM  [No.22DZ2229014].}}

\date{\today}

\begin{document}
\maketitle

\begin{abstract}
This study examines    nonnegative solutions to the problem
\begin{equation*}\left\{\arraycolsep=1.5pt
\begin {array}{lll}
\Delta u=\displaystyle\frac{\lambda|x|^{\alpha}}{u^p} \   \ &\hbox{ in} \,\  \R ^2\setminus \{0\},\\[2mm]
u(0)=0 \ \text{and}\ u> 0 \   \ &\hbox{ in} \,\  \R ^2\setminus \{0\},\\
\end{array}\right.
\label{eqn}
\end{equation*}
%
where $\lam >0,$ $\alp>-2$, and $p>0$ are  constants.
The possible  asymptotic behaviors  of $u(x)$ at $|x|=0$ and $|x|=\infty$  are classified according to $(\alpha,p)$. In particular, the results show that for some $(\alpha,p)$, $u(x)$ exhibits only  ``isotropic" behavior at $|x|=0$ and $|x|=\infty$. However, in other cases, $u(x)$ may exhibit the ``anisotropic" behavior at $|x|=0$ or  $|x|=\infty$.
Furthermore, the relation between the limit at $|x|=0$ and the limit at  $|x|=\infty$ for a global solution is investigated.
\end{abstract}
\vskip 0.2truein




\vskip 0.2truein

\section{Introduction}


In this paper, we  
consider the following nonlinear problem
\begin{equation}\left\{\arraycolsep=1.5pt
\begin {array}{lll}
\Delta u=\displaystyle\frac{\lambda|x|^{\alpha}}{u^p} \   \ &\hbox{ in} \,\  \R ^2\setminus \{0\},\\[2mm]
u(0)=0 \ \text{and}\ u> 0 \   \ &\hbox{ in} \,\  \R ^2\setminus \{0\},\\
\end{array}\right.
\label{eqn}
\end{equation}
where $\lam >0$, $\alp>-2$, and $p>0$ are constants.

Equations of the form in \eqref{eqn} with a negative exponent arise in several applied problems, such as
 in the study of steady states of thin films of viscous fluids (see, for example, \cite{MR2326181}),
and in the modeling of electrostatic micro-electromechanical systems (MEMS) (refer to \cite{MR1955412} for physical derivations and \cite{MR3662914,MR2604963} for mathematical analysis). In particular, $|x|^{\alpha}$ represents the power-law permittivity profile in MEMS modeling. Consequently, the elliptic equation in \eqref{eqn} has been widely investigated over the past few years (see \cite{MR2161885,MR4213008,MR3265537,MR2500854,MR2393396,MR2498845,MR3436432,MR3023006,MR2604963,MR2286013,MR3426132,MR2179754,MR2564407,MR4375793,MR4631976,MR2351179,MR2262209} and  references therein), while most  studies have focused on positive solutions.

Of special interest are solutions that give rise to singularities in the equation, that is, such that $u=0$ in some regions. These solutions are called rupture solutions because they represent  rupture in the device  in the physical model. The parabolic version related to  \eqref{eqn} was considered in \cite{kawarada1975solutions,MR3867223,MR3945769,MR3305368,MR2179754,MR4389601}. In such a setting, a solution $u$  that takes the value $0$ at some point is called a touchdown solution.

 In this study, for the elliptic equation in \eqref{eqn}, we consider its rupture solutions in the simplest case, where the rupture point is the origin. For $\alpha=0,$ the rupture solutions for equation  \eqref{eqn} in dimension $N$ were investigated in \cite{MR2957549,MR3265537,MR3436432,MR2500854,MR2393396,MR2161885,MR4375793}. In particular, the H$\ddot{o}$lder continuity of rupture solutions and the Hausdorff dimensions of  rupture sets were studied in \cite{MR3436432};  the existence of a single point rupture solution in some $R^N, N\geq3$ was obtained in \cite{MR2161885}; for \eqref{eqn} in $\R^N$ with $N\geq3,$ infinitely many non-radial rupture solutions were constructed in \cite{MR3265537} for some $p$. For general $\alpha$, the radial single point rupture solution near the origin for \eqref{eqn} in dimension $N$ was recently considered  in \cite{MR4375793,MR4631976,MR4213008}.

In fact, the behaviors of the   solutions  near the origin were described in \cite{MR3998250} for problem \eqref{eqn} in a special case where $p=2$. The second author of the present article and her collaborators conducted this study. More precisely, in \cite{MR3998250}, the  ``isotropic" and ``anisotropic" ruptures  at the origin of the solution for \eqref{eqn} were classified with respect to $\alpha$ for $p=2.$

 In the present study,  we extend the conclusions of \cite{MR3998250} from  $p=2$  to general $p>0$ and expand the domain to the entire $\R^2.$ Compared with \cite{MR3998250}, first, instead of only one $\alpha$, this study contains two parameters, $\alpha$ and $p$,   which yield richer results.  Second, we will further study global solutions on  $\R ^2\setminus \{0\}$, which is the main concern of this paper and was not  investigated in \cite{MR3998250}.  In addition, our results show a fundamental difference for the case of $p=3$    compared with the other cases.

Before describing  the outcome of this study precisely,  we  rewrite  equation in \eqref{eqn} as
\beq\label{62}
u_{rr}+\frac{1}{r}u_r+\frac{1}{r^2}u_{\theta\theta}=\frac{\lambda r^{\alpha}}{u^p}, \quad (r, \theta)\in [0,+\infty)\times S^1,
\eeq
where $u(x)=u(r, \theta)$ and $(r, \theta)$ is the polar coordinate. Furthermore, we define $v(t, \theta)$ by
\begin{equation}\label{in:trans}
v(t,\theta ):=r^{-\frac{\alp+2}{p+1}}u(r,\theta), \ \ \mbox{where}\,\ t=\ln r \ \ \mbox{and}\ \ r=|x|.
\end{equation}
Denoting
\beq\label{85}
\beta:=\frac{\alp+2}{p+1},
\eeq
 $v(t,\theta)$ satisfies the following elliptic evolution  problem
\begin{equation}\label{R:trans-1}
  -v_{tt}- 2\beta v_t=v_{\theta\theta}+ \beta^2 v-\displaystyle\frac{\lam}{v^p}, \quad (t,\theta)\in (-\infty,+\infty)\times S^1.
\end{equation}
In association with the stationary problem of \eqref{R:trans-1}, we also define $w(\theta)$ as a solution to
\beq\label{2.6}
w''+\beta^2w-\frac{\lambda}{w^p}=0  \ \ \hbox{on}\,\     S^1,
\eeq
and denote  the set of all solutions  as
\beq\label{2.6S}
\mathfrak{S}=\big\{w>0:\, w \ \mbox{is a solution of}\ \eqref{2.6} \big\}.
\eeq

In this study,  we focus on  the rupture solutions $u$ of \eqref{eqn} satisfying
\beq \label{81}
0<C_1\leq r^{-\frac{\alp+2}{p+1}}u(r,\theta)\leq C_2<\infty,
\eeq
which is equivalent to the solutions
 $v$ of \eqref{R:trans-1}  satisfying
\beq\label{C14}
0<C_1\leq v\leq C_2<\infty,
\eeq
where $C_1$ and $ C_2$ are constants independent of $(t,\theta)$ but  dependent on $v.$ Denote
\beq\label{86}
m_0:=\left(\frac{\lambda}{\beta^2}\right)^{\frac{1}{p+1}}=\left(\frac{\lambda(p+1)^2}{(\alp+2)^2}\right)^{\frac{1}{p+1}}.
\eeq
Clearly,  the trivial solution $v\equiv m_0$ for \eqref{R:trans-1} satisfies the requirement of  \eqref{C14}. Nevertheless, it should be noted that some solutions for equation \eqref{R:trans-1}  may not meet the requirements of equation \eqref{C14}. For instance, following  Guo and Wei   \cite{MR3265537}, we can construct a solution $v=[\frac{\lambda(p+1)^2}{2(1-p)}]^{\frac{1}{p+1}}(\sin\theta)^{\frac{2}{p+1}}$  for \eqref{R:trans-1} with $\alpha=0,0<p<1$, which does not satisfy \eqref{C14}.

We are now in a position to state our first main result concerning  the structure of $\mathfrak{S}$ in terms of $\alpha$ and $p$, which plays a fundamental role in the classification of solution behaviors.

\begin{figure}[ht]

\centering
\includegraphics[scale=0.4]{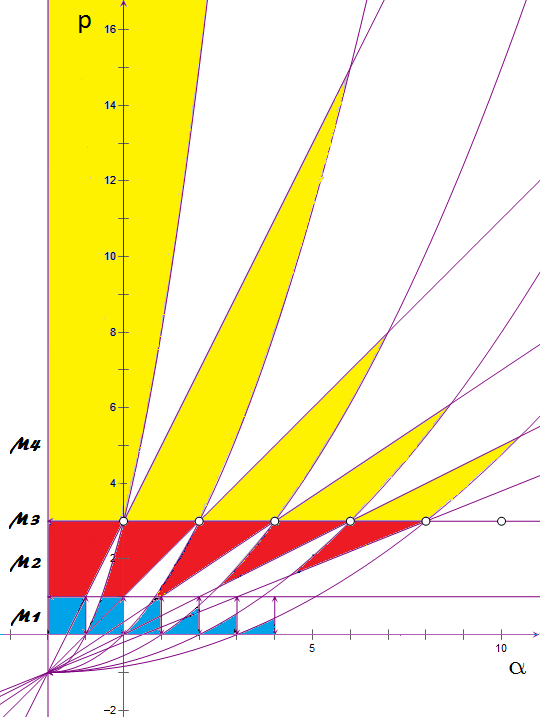}
\caption{Set $\mathscr{M}:=\mathscr{M}_1\cup \mathscr{M}_2\cup \mathscr{M}_3\cup \mathscr{M}_4$ }
\label{fig:1}
\end{figure}

\begin{thm}\label{thm1.1} 

Consider the  set
\begin{equation} \label{thm1.1:00}
\mathfrak{S}:=\Bigg\{w(\theta )\in C^2(S^1):\, w''+\bigg(\frac{\alp+2}{p+1}\bigg)^2w-\frac{\lam}{w^p} =0,\quad w>0 \Bigg\}.
\end{equation}
as defined in \eqref{2.6S}, where $\alpha>-2, ~p>0, ~ \lam  >0$  are given constants. The results are as follows:
\begin{enumerate}
\item[$(1)$] If
\begin{equation} \label{thm1.1:1}
(\alpha,p)\in\mathscr{M}:=\mathscr{M}_1\cup \mathscr{M}_2\cup \mathscr{M}_3\cup \mathscr{M}_4,
\end{equation}
 then
 $\mathfrak{S} =\big\{\big(\frac{\lam}{\beta^2}\big)^{\frac{1}{p+1}}\big\}.$
Here  (refer to Figure \ref{fig:1})
\begin{equation}
  \begin{split}
  \mathscr{M}_1 &:=\bigg\{(\alpha,p)|~\alpha>-2,\,\ 0<p<1\}\cap\{(\alpha,p)|~\alpha\leq -1 \\
  & \ \text{or}\ \exists ~j\in \N, \,\ j\geq 1,\,\ \st\,\ \alpha\leq j-1 \ \text{and}\ p\leq \left(\frac{\alpha+2}{j}\right)^2-1 \bigg\},\\
  \mathscr{M}_2 &:=\bigg\{(\alpha,p)|~\alpha>-2,\,\ 1\leq p<3\}\cap\{(\alpha,p)\big|~p\geq 2\alpha+3 \\
  &\ \text{or}\ \exists ~j\in \N, \,\ j\geq 1,\,\ \st\,\ \frac{2(\alpha+2)}{j+1}-1\leq p\leq \left(\frac{\alpha+2}{j}\right)^2-1\bigg\},\\
  \mathscr{M}_3 &:=\bigg\{(\alpha,3)|~\alpha>-2,\,\ \text{and}\,\ \alp \notin 2\N\bigg\},\\
  \mathscr{M}_4 &:=\bigg\{(\alpha,p)|~\alpha>-2,\,\ p>3\}\cap\{(\alpha,p)\big|~p\geq  (\alpha+2)^2-1 \\
  & \ \text{or}\ \exists ~j\in \N, \,\ j\geq 1,\,\ \st\,\ \left(\frac{\alpha+2}{j+1}\right)^2-1 \leq p\leq \frac{2(\alpha+2)}{j}-1 \bigg\}.\notag
  \end{split}
\end{equation}

\item[$(2)$] If $(\alpha,p)\in\{(\alp,p)|~\alpha>-2,\,\ p>0,p\neq3 \}\backslash{\mathscr{M}}$, then $\mathfrak{S}$ contains  exactly $1+N_0(\alp,p)$ connected components  $\mathfrak{S}_0=\big\{\big(\frac{\lam}{\beta^2}\big)^{\frac{1}{p+1}}\big\}$,  $\mathfrak{S}_1,\cdots ,\mathfrak{S}_i,$ $\cdots ,\mathfrak{S}_{N_0(\alp,p)}$. Here
    \beq \label{thm1.1:2}
\mathfrak{S}_i=\big\{w_{j_i}(\cdot+a);0\leq a<2\pi\big\},\,\ i=1,2,\cdots,N_0(\alp,p),
\eeq
 $w^{j_i}(\theta)$ is the $\frac{2\pi}{{j_i}}-$periodic positive solution of \eqref{2.6},
$j_i=[b_{\alp,p}]+i,$, $1\leq N_0(\alp,p)< +\infty$ denotes the number of integers in
\[
J_{\alpha,p}=
\begin {cases}
\big(\frac{\alpha+2}{\sqrt{p+1}},~\alpha+2\big) \ & 0<p<1,\\
\big(\frac{\alpha+2}{\sqrt{p+1}},~\frac{2(\alpha+2)}{p+1}\big) \ & 1\leq p<3,\\
\big(\frac{2(\alpha+2)}{p+1},~\frac{\alpha+2}{\sqrt{p+1}}\big) \ & p>3.
\end{cases}
\]
 and $b_{\alp,p}$ is the left endpoint of $J_{\alpha,p}$.
\item[$(3)$] If $p=3$ and $\alpha\in 2\N,$ then $\mathfrak{S}$ contains exactly  one connected component. More precisely, it holds $\mathfrak{S}=\{w_{\epsilon,a}(\theta);(\epsilon,a)\in (0,1]\times[0,\pi/\beta)\}$ with $\beta=(\alpha+2)/4$ and
    \beq\label{8}
    w_{\epsilon,a}(\theta)=\big(\frac{\lam}{\beta^2}\big)^{\frac{1}{4}}\left[\epsilon\cos^2(\beta(\theta+a))+\epsilon^{-1}\sin^2(\beta(\theta+a))\right]^{\frac{1}{2}}. 
    \eeq
We point out that the  minimum half-period of $w_{\epsilon,a}(\theta)$ is  $\frac{2}{\alpha+2}\pi$ for all $(\epsilon,a)\in (0,1]\times[0,\pi/\beta)$.
Unlike the case $p\neq3$, this connected component contains the trivial solution $w_{1,a}\equiv\big(\frac{\lam}{\beta^2}\big)^{\frac{1}{4}}$.

\end{enumerate}
\end{thm}

\begin{rem}
The set  $\mathscr{M}$ depends only on $\alpha$ and $p,$ and is not affected by $\lambda.$
\end{rem}
\begin{rem}\label{rem1.3} Because  $\mathscr{M}$ contains  an infinite number of parts, only a finite number  is shown in Figure \ref{fig:1}.
\end{rem}

\begin{rem}\label{rem1.1} For the particular case $\alpha=0$(which is often discussed in the literature, e.g., \cite{MR3436432,MR2604963,MR2500854,MR3265537}), equations \eqref{eqn} and \eqref{2.6} can be rewritten as
\begin{equation}\label{137}\left\{\arraycolsep=1.5pt
\begin {array}{lll}
\Delta u=\displaystyle\frac{\lambda}{u^p} \   \ &\hbox{ in} \,\  \R ^2\setminus \{0\},\\[2mm]
u(0)=0 \ \text{and}\ u> 0 \   \ &\hbox{ in} \,\  \R ^2\setminus \{0\},\\
\end{array}\right.
\end{equation}
and
\beq\label{138}
w''+\left(\frac{2}{p+1}\right)^2w-\frac{\lambda}{w^p}=0  \ \ \hbox{on}\,\     S^1,
\eeq
respectively.  According to Theorem \ref{thm1.1}(1), now $${\mathscr{M}} =\{(0,p)| p\in(0,3)\cup(3,+\infty)\}.$$
Therefore, for all $p\neq 3$, it holds $\mathfrak{S} =\big\{\big(\frac{\lam (p+1)^2}{4}\big)^{\frac{1}{p+1}}\big\},$ that is, equation \eqref{138}  has only trivial solution.
For $p=3$, according to Theorem \ref{thm1.1}(3), \eqref{138} has a family of solutions:
$$w_{\epsilon,a}(\theta)=\sqrt{2}\lambda^{\frac{1}{4}}\left[\epsilon\left(\cos{\frac{\theta+a}{2}}\right)^2+\epsilon^{-1}\left(\sin{\frac{\theta+a}{2}}\right)^2\right]^{\frac{1}{2}},\ (\epsilon,a)\in (0,1]\times[0,2\pi).$$
This implies that there exists a family of non-radial rupture solutions for \eqref{137} $u(r,\theta)=\sqrt{r}w_{\epsilon,a}(\theta),$
which is consistent with the results presented in \cite{MR3265537}.
\end{rem}
\begin{rem}
For the particular case $p=2$ (where the equation in \eqref{eqn} is the standard MEMS equation), according to Theorem \ref{thm1.1},  it holds that $${\mathscr{M}} =\displaystyle\bigg\{(\alpha,2)\bigg| \alpha\in\big(-2,-{1\over2}\big]\cup \displaystyle\bigcup^7_{k=2}\bigg[(k-1)\sqrt 3-2,\frac{3k-4}{2}\bigg]\bigg\},$$ consistent with the result presented in \cite[Theorem 1.2]{MR3998250}.
\end{rem}
\begin{rem}
For a better understanding of Theorem \ref{thm1.1}(2), we present the following examples.  If $p=2$ and $\alpha \in (1,2\sqrt{3}-2)$, then $\big[\frac{\sqrt{3}(2+\alpha)}{3}\big]=1$, $N_0(\alp)=1,$ $j_1=2,$ $\mathfrak{S}_1=\big\{w_2(\cdot+a):\, a\in S^1\big\}$, and $\mathfrak{S}=\mathfrak{S}_0\cup\mathfrak{S}_1=\big\{ \big(\frac{9\lam}{(2+\alp )^2}\big)^{\frac{1}{3}}\big\}\cup\big\{w_2(\cdot+a):\, a\in S^1\big\}$;
if $p=2$ and  $ \alpha \in (\frac{5}{2},3\sqrt{3}-2)$, then
$\big[\frac{\sqrt{3}(2+\alpha)}{3}\big]=2$, $N_0(\alp)=1,$ $j_1=3,$ $\mathfrak{S}_1=\big\{w_3(\cdot+a):\, a\in S^1\big\}$, and $\mathfrak{S}=\mathfrak{S}_0\cup\mathfrak{S}_1=\big\{ \big(\frac{9\lam}{(2+\alp )^2}\big)^{\frac{1}{3}}\big\}\cup\big\{w_3(\cdot+a):\, a\in S^1\big\}$.
\end{rem}
\begin{rem}
Equation \eqref{2.6} also arises from the study of the generalized curve-shortening problem. In particular, for $p=3,$ equation \eqref{2.6} with $\beta=1$ is called the affine curve-shortening problem (see \cite{MR1949167,MR2796243}).
\end{rem}

Next, we investigate the asymptotic behavior of the solution $u$ to \eqref{eqn} at the origin and  infinity.

\begin{thm}\label{thm1.2} If $u(x)=u(r,\theta)$ is a solution of \eqref{eqn} and satisfies condition \eqref{81}, then there exist $w_0, w_{\infty}\in \mathfrak{S}$ such that
\begin{align}
 \label{135} \|r^{-\frac{\alp+2}{p+1}}u(r,\theta)-w_0(\theta)\|_{C^2(S^1)} &\leq C_1(1-lnr)^{-\frac{s_1}{7(1-2s_1)}} \quad \text{as}\ r\rightarrow 0^+,\\
 \label{136} \|r^{-\frac{\alp+2}{p+1}}u(r,\theta)-w_{\infty}(\theta)\|_{C^2(S^1)} &\leq C_2(1+lnr)^{-\frac{s_2}{7(1-2s_2)}} \quad \text{as}\ r\rightarrow +\infty,
\end{align}
for some $s_1, s_2 \in(0,\frac{1}{2})$  depending on $w_0, w_{\infty}$, respectively. Here $C_1, C_2$ are constants.
\end{thm}

By combining Theorem \ref{thm1.1}  and Theorem \ref{thm1.2}, the following conclusion is immediately obtained.
\begin{cor}\label{139}
Suppose $u(x)=u(r,\theta)$ is a solution to \eqref{eqn} and satisfies condition \eqref{81}.
\begin{enumerate}

 \item[$(1)$] If 
 $(\alpha,p)\in{\mathscr{M}},$ then
\begin{equation}\label{140}
  \lim _{r\to 0^+}r^{-\frac{2+\alp}{p+1}} u(r,\theta)= \lim _{r\to +\infty}r^{-\frac{2+\alp}{p+1}} u(r,\theta)=  \left(\frac{\lambda(p+1)^2}{(\alp+2)^2}\right)^{\frac{1}{p+1}} \quad \text{in}\ \ C^2(S^1).
\end{equation}

 \item[$(2)$]
      If $(\alpha,p)\in\{(\alp,p)|~\alpha>-2,\,\ p>0,p\neq3 \}\backslash{\mathscr{M}}$, then either \eqref{140} holds
 or there exist $w_{j_i}(\theta+a_0)$ and $w_{j_k}(\theta+a_{\infty}),$$1\leq i,k\leq N_0(\alpha,p),$ $a_0,a_{\infty}\in[0,2\pi)$ such that

\begin{equation}\label{141}
 \lim _{r\to 0^+}r^{-\frac{2+\alp}{p+1}}u(r, \theta)= w_{j_i}(\theta+a_0)\ \text{and}\  \lim _{r\to +\infty}r^{-\frac{2+\alp}{p+1}}u(r, \theta)= w_{j_k}(\theta+a_{\infty})
\end{equation}
 in $C^2(S^1).$

 \item[$(3)$] If $p=3$ and $\alpha\in 2\N,$ then there exist $(\epsilon_0,a_0),(\epsilon_{\infty},a_{\infty})\in (0,1]\times[0,2\pi)$ such that
 \begin{equation}\label{142}
 \lim _{r\to 0^+}r^{-\frac{2+\alp}{p+1}}u(r, \theta)=w_{\epsilon_0,a_0}(\theta)\ \text{and}\  \lim _{r\to 0^+}r^{-\frac{2+\alp}{p+1}}u(r, \theta)=w_{\epsilon_{\infty},a_{\infty}}(\theta)
\end{equation}
 in $C^2(S^1).$
\end{enumerate}
\end{cor}
Based on  previous preparations, we   focus on the global rupture solutions,  that is,   functions  $u\in C^2(\R^2\setminus \{0\})$  satisfying equation \eqref{eqn} and condition \eqref{81}. This is the main concern of this study.
According to  Theorem \ref{thm1.2},  the limit functions
\begin{equation}\label{120}
\begin{split}
  w_0(\theta) &=\lim\limits_{r \to 0}r^{-\frac{\alp+2}{p+1}}u(r,\theta)=\lim\limits_{t \to -\infty}v(t,\theta),\\
  w_{\infty}(\theta) &=\lim\limits_{r \to +\infty}r^{-\frac{\alp+2}{p+1}}u(r,\theta)= \lim\limits_{t \to +\infty}v(t,\theta),
\end{split}
\end{equation}
both belong to $\mathfrak{S},$ which is defined  in \eqref{2.6S}. Here, $v=r^{-\frac{\alp+2}{p+1}}u$ is a solution of  \eqref{R:trans-1}.
Our main concern  is the  relation between $w_0(\theta)$ and $ w_{\infty}(\theta)$. In particular, we want to know
that if  there exist global solutions satisfying \eqref{120}, then what necessary conditions the pair of elements $w_0(\theta), w_{\infty}(\theta)$  must satisfy. In other words, we want to investigate for which pairs  of elements $w_0(\theta), w_{\infty}(\theta)$ there may exist global solutions satisfying \eqref{120},
and for which pairs there is no global solution satisfying \eqref{120}.

Inspired by \cite{MR0993442}, a key observation is derived from an estimate of energy: the energy at the origin is always less than or equal to that at  infinity. More precisely,  the energy functional $E$ is defined as
\begin{equation}\label{E0}
E(\psi)=\int_{S^1}~ \frac{1}{2}(\psi')^2-\frac{1}{2}\beta^2\psi^2-h(\psi)~d\theta,
\end{equation}
where $h$ is defined by \eqref{C15}.  If $w_0(\theta), w_{\infty}(\theta)$  are defined as in \eqref{120}, then we have  (Proposition \ref{3.prop1})
\beq\label{118}
E(w_0)\leq E(w_{\infty}).
\eeq
Here, the equality sign holds if and only if $w_0(\theta)= w_{\infty}(\theta)$  and $u$ is of the form $u(r,\theta)=r^{\frac{\alp+2}{p+1}}w(\theta),$  $w(\theta)\in\mathfrak{S}$. By Theorem \ref{thm1.1}, this implies that if $(\alpha,p)$ satisfies \eqref{thm1.1:1}, then
 $u(r,\theta)=\big(\frac{\lam}{\beta^2}\big)^{\frac{1}{p+1}}r^{\frac{\alp+2}{p+1}}$ is the unique global solution to \eqref{eqn} that satisfies \eqref{81} (Corollary \ref{83}).

 For the special case $p=3$, it holds that $E(w)\equiv C$ on $  \mathfrak{S},$ where $C$ is a constant independent of $w$ (Proposition \ref{110}). Consequently,  $u$ has the form $u(r,\theta)=r^{\frac{\alp+2}{p+1}}w(\theta),$  $w(\theta)\in\mathfrak{S}$ (Corollary \ref{0805}). To be more precise,
there  holds $u(r,\theta)=r^{\frac{\alp+2}{p+1}}\big(\frac{\lam}{\beta^2}\big)^{\frac{1}{p+1}}$ for $\alp \notin 2\N$ and $u(r,\theta)=r^{\frac{\alp+2}{p+1}} w_{\epsilon,a}(\theta)$ for $\alp \in 2\N$, where $w_{\epsilon,a}(\theta)$ is as in \eqref{8}.

For the rest of the cases, it is noted that   the energy $E$ is constant on each connected component of $\mathfrak{S}.$
Furthermore, the analytical results in Section \ref{117}  and the numerical monotonicity results in Section \ref{134} for $E$ indicate

\begin{equation}\label{119}
\begin {array}{lll}
  E(\mathfrak{S}_0)<E(\mathfrak{S}_1)<E(\mathfrak{S}_2)<\cdots <E(\mathfrak{S}_j)<\cdots <E(\mathfrak{S}_{N_0}),\quad \text{for}\ 0<p<3,\\
  E(\mathfrak{S}_1)<E(\mathfrak{S}_2)<\cdots <E(\mathfrak{S}_j)<\cdots <E(\mathfrak{S}_{N_0})<E(\mathfrak{S}_0),\quad \text{for}\ p>3,
  \end{array}
  \end{equation}
where $\mathfrak{S}_k,k=1,2,\ldots,N_0$ are defined  in Theorem \ref{thm1.1}.
From \eqref{118},  inequality \eqref{119} can be interpreted as follows: the asymptotic profile of a global solution at the origin has a lower (or equal) frequency than that at infinity.

In conclusion, we have the following two main theorems:
\begin{thm} \label{121}Let $\mathscr{M},w_{\epsilon,a}(\theta)$ be defined as in Theorem \ref{thm1.1}.
\begin{enumerate}

 \item[$(1)$]If $(\alpha,p)\in\mathscr{M}$, then
 $u(r,\theta)=\big(\frac{\lam}{\beta^2}\big)^{\frac{1}{p+1}}r^{\frac{\alp+2}{p+1}}$ is the unique global solution to \eqref{eqn} that satisfies \eqref{81}.

 \item[$(2)$] If $p=3$ and $\alp \in 2\N$, then $u(r,\theta)$ is a global solution to \eqref{eqn} that satisfies \eqref{81} if and only if
  $u(r,\theta)=r^{\frac{\alp+2}{p+1}} w_{\epsilon,a}(\theta)$ for some $(\epsilon,a)\in (0,1]\times[0,\pi/\beta).$
\end{enumerate}
 \end{thm}
\begin{rem}
  The conclusion of Theorem \ref{121} is stronger than Corollary \ref{139}(1)(3).
\end{rem}

Next, for the remaining cases, we provide a necessary condition that $(w_0(\theta),w_{\infty}(\theta)$  must satisfy in order to have a global solution that satisfies $\eqref{120}.$

\begin{thm}\label{thm1.3} Let $(\alpha,p)\notin\mathscr{M},$ $p>0,p\neq3,$ $\mathfrak{S}_k,k=1,2,\ldots,N_0$ be defined as in Theorem \ref{thm1.1}.
Assume that $u(x)=u(r,\theta)$ is a global rupture solution of \eqref{eqn}, satisfying \eqref{81} and \eqref{120} with
$w_0(\theta) \in \mathfrak{S}_m$,$w_{\infty}(\theta) \in \mathfrak{S}_n$. If \eqref{119} holds, then
\begin{enumerate}
  \item[$(1)$] for $1<p<3$, it holds  $0\leq m\leq n$;
  \item[$(2)$]  for $p>3$, either $0< m\leq n$, or $n=0,$ $0\leq m\leq N_0$.
\end{enumerate}
\end{thm}

\begin{rem}
Each of the above theorems and its corollary show that  $p=3$ is  fundamentally different  from the other cases.
\end{rem}

Before concluding this section, we want to stress some new features of our study.

(1) For solutions that satisfy  \eqref{eqn} and \eqref{81}, Figure  \ref{fig:1} completes the classification of all possible rupture behaviors concerning $(\alpha,p)$ at $|x|=0$ and  $+\infty$. Specifically,
 it only admits  the ``isotropic" rupture behavior  in the sense of \eqref{140} for $(\alpha,p)$ in the colored region. Otherwise, it also  allows for the ``anisotropic" rupture behavior  in the sense of \eqref{141} or \eqref{142}.

(2) Our research  reveals a remarkable discovery: when  $p=3$,  it exhibits distinctive characteristics compared with the other cases.  First, it should be noted from Theorem \ref{thm1.1} that  when  $p=3$, $\mathfrak{S}$ remains connected for all $\alpha>-2$, unlike when $0<p\neq3$, $\mathfrak{S}$ is not   connected for most $\alpha>-2$. Second, also by  Theorem \ref{thm1.1}, for  $p=3$, all  nontrivial functions in $\mathfrak{S}$ have the same frequency for all $\alpha>-2$, which is not true for $0<p\neq3$ . Finally, but most importantly, according to Theorem \ref{121}, the global solutions for  $p=3$  behave in a simpler manner than  the other cases.

(3) It is worth noting that elliptic integrals were employed in our study.  In fact, we were unable to determine the asymptotic behavior of $E$ for $p\in(1,3)$ until we established a connection with elliptic integrals.

The remainder of this paper is organized as follows. Section 2 analyzes the structure of set $\mathfrak{S}$ and presents the proof of Theorem \ref{thm1.1}. Then, in Section 3, we examine the asymptotic behaviors of the solution $u$ near the origin and infinity, and provide a proof for Theorem \ref{thm1.2}. The final section, Section 4, focuses on studying the global solutions and presenting proof for Theorems \ref{121} and \ref{thm1.3}.

\section{The structure of set $\mathfrak{S}$}\label{3.1}
In this section, we   examine the set $\mathfrak{S}$   defined in \eqref{thm1.1:00} and apply the  phase-plane method to prove Theorem \ref{thm1.1}. For convenience, we define:
\beq \label{C15}
h(x)=\left\{\arraycolsep=1.5pt
\begin {array}{lll}
\frac{\lambda}{(p-1)x^{p-1}}, \ &\text{for}\ p\neq1,\\
-\lambda\ln x, \ &\text{for}\  p=1.
\end{array}\right.
\eeq
Multiplying \eqref{2.6} by $w'$ and then integrating over $S^1$ yields
\beq \label{11}
(w')^2+{\beta}^2w^2+ 2h(w)=E,
\eeq
where $E$ is a constant.
 Define
\beq\label{lll}
g(w)={\beta}^2w^2+ 2h(w),\quad w>0.
\eeq
Then \eqref{11} can be rewritten as
\[
(w')^2=E-g(w).
\]
Note that $w_0=\big(\frac{\lambda}{{\beta}^2}\big)^{1/(p+1)}$ is the unique critical point of $g$ and
\[ \label{15}
 g(w_0):=E_0 =\left\{\arraycolsep=1.5pt
\begin {array}{lll}
\frac{p+1}{p-1}\lambda^{\frac{2}{p+1}}{\beta}^{\frac{2(p-1)}{p+1}}, \ &\text{for}\ p\neq1\\
\lambda(1-\ln{\lambda\over {\beta^2}}), \ &\text{for}\  p=1
\end{array}\right..
\]
Furthermore, $g(w)$ decreases monotonically on $(0, w_0)$ and increases monotonically on $(w_0, +\infty)$. Consequently, for $p > 0$, \eqref{2.6} has a nontrivial solution if and only if $E> g(w_0)$. In addition, any nontrivial solution of problem \eqref{2.6} has the following two properties: (i) it is periodic; (ii) if $w(\theta)$ is a solution of \eqref{2.6}, then $w(\theta+a)$ is also a solution of \eqref{2.6} for any $a\in\mathbb{R}.$

 Suppose that $w(\theta)$ is a nontrivial positive solution to \eqref{2.6}. Denote $w_1$ (resp. $w_2$) the minimum (resp. maximum) value of $w(\theta).$ Then, $w_1$ and $w_2$ are the two roots of
\[
g(w)=E \quad \text{for}\,\ E> g(w_0).
\]
Therefore, by setting $\tau=\frac{w_2}{w_1}$, we conclude from the above that
\beq\label{9}
\left\{\arraycolsep=1.5pt
\begin {array}{lll}
w_1^{p+1 }=\frac{2\lambda(1-{\tau}^{1-p})}{{\beta}^2(p-1){(\tau^2-1)}}, \ &\text{for}\ p\neq1,\\
w_1^2=\frac{2\lambda\ln\tau}{\beta^2(\tau^2-1)}, \ &\text{for}\  p=1.
\end{array}\right.
\eeq
We point out that by ODE theory, for any given $\tau\in[1,+\infty),$ there exists a unique (up to shift) solution $w(\theta)$ to \eqref{2.6} such that $w_2/w_1=\tau.$

Without loss of generality, we can assume that $\theta=0$ is a minimum point of $w(\theta)$ and $\theta =L>0$ is a maximum point of $w(\theta)$, such that $w'(\theta)>0$ for any $\theta\in(0,L).$ Therefore, $w'(0)=w'(L)=0,$ where $L>0$ is the minimum half-period of $w.$ Note also from \eqref{11} that
\begin{equation}\label{90}
d\theta=\frac{dw}{\sqrt{E-g(w)}},
\end{equation}
which implies that
\beq\label{803}
L=\int_{w_1}^{w_2}\frac{dw}{\sqrt{E-g(w)}}. 
\eeq
By setting
$y=\frac{w}{w_1}$, \eqref{803} can be rewritten as

\begin{equation}
L(\tau)           = \int_1^\tau \frac{dy}{\sqrt{\frac{E}{w_1^2}-{g(w_1y)\over{w_1^2}}}}
           = \left\{\arraycolsep=1.5pt
\begin {array}{lll}\displaystyle
\frac{1}{\beta}\int_1^\tau\frac{dy}{\sqrt{1+\frac{\tau^{p-1}(\tau^2-1)}{\tau^{p-1}-1}-y^2-\frac{\tau^{p-1}(\tau^2-1)}{\tau^{p-1}-1}{1\over{y^{p-1}}}}},\quad &p\neq 1,\\
\displaystyle\frac{1}{\beta}\int_1^\tau\frac{dy}{\sqrt{1-y^2+(\tau^2-1)\frac{\ln y}{\ln \tau}}}, \quad &p=1,
\end{array}\right.
\label{14}
\end{equation}
by \eqref{C15}, \eqref{lll} and \eqref{9}. Next, we analyze the range of $L(\tau)$.

\begin{lem}\label{31}
$L(\tau)$ is continuous on $[1,+\infty)$ and satisfies
\begin{equation}\label{12}
\begin{split}
\lim_{\tau\rightarrow 1}L(\tau)&=\frac{\pi}{\sqrt{p+1}\beta}=\frac{\sqrt{p+1}}{\alpha+2}\pi,\\
\lim_{\tau\rightarrow +\infty}L(\tau)&=\frac{\pi}{\min\{p+1,2\}\beta}=\frac{p+1}{\min\{p+1,2\}(\alpha+2)}\pi.\\
\end{split}
\end{equation}
Moreover, $L(\tau)$ decreases strictly for $p\in(0,3)$, increases strictly for $p\in(3,+\infty)$,  and $L(\tau)\equiv\frac{2\pi}{\alpha+2}$ for $p=3$.
\end{lem}

\proof We denote $Q(w)=-\beta^2w+\frac{\lambda}{w^p}$, and let $w_0=\left(\frac{\lambda}{{\beta}^2}\right)^{\frac{1}{p+1}}$ be  the unique root of $Q(w).$
From \cite[Lemma 3.2]{MR0727393}, we  have
\beq \label{3:49}
L(E)\xlongrightarrow{E\rightarrow E_0 }\frac{\pi}{\sqrt{-Q'(w_0)}}=\frac{\pi}{\sqrt{p+1}\beta}, \ \text{for}\ p\in(0,+\infty).
\eeq
Thus, the first equation of \eqref{12} follows directly from (\ref{3:49}) because $\tau\rightarrow 1$ is equivalent to $E\rightarrow E_0$. By setting $\xi=\frac{y-1}{\tau-1}$, we rewrite \eqref{14} as
\begin{equation}\label{112}
L(\tau) =\frac{1}{\beta}\int_0^1\frac{\tau-1}{\sqrt{1+\frac{\tau^{p-1}(\tau^2-1)}{\tau^{p-1}-1}-(\xi(\tau-1)+1)^2-\frac{\tau^{p-1}(\tau^2-1)}{\tau^{p-1}-1}{1\over{(\xi(\tau-1)+1)^{p-1}}}}}d\xi \end{equation}
for $ p\neq 1,$
and
\[
L(\tau) =\frac{1}{\beta}\int_0^1\frac{\tau-1}{\sqrt{1-(\xi(\tau-1)+1)^2+(\tau^2-1)\frac{\ln{(\xi(\tau-1)+1)}}{\ln \tau}}}d\xi
\]
for $ p=1.$
We then obtain

\begin{equation}
\lim_{\tau\rightarrow +\infty}L(\tau)=
\left\{\arraycolsep=1.5pt
\begin {array}{lll}
\frac{1}{\beta}\int_0^1\frac{d\xi}{\sqrt{1-\xi^2}}=\frac{\pi}{2\beta}, \ &\text{for}\ p\in[1,+\infty),\\
\frac{1}{\beta}\int_0^1\frac{d\xi}{\sqrt{\xi^{1-p}-\xi^2}}=\frac{\pi}{(p+1)\beta},  \ &\text{for}\ p\in (0,1].\\
\end{array}\right.
\end{equation}
This indicates the validity of the second equation in \eqref{12}. For $p=3,$ it can be calculated that
\begin{equation}
\begin{split}
L(\tau)=&\frac{1}{\beta}\int_0^1~\frac{\tau-1}{\sqrt{1+\tau^2-(\xi(\tau-1)+1)^2-\frac{\tau^2}{(\xi(\tau-1)+1)^2}}}~d\xi\\
\xlongequal{x=(\xi(\tau-1)+1)^2}&\frac{1}{\beta}\int_1^{\tau^2}\frac{1}{2}\frac{dx}{\sqrt{(\frac{\tau^2-1}{2})^2-(x-\frac{\tau^2+1}{2})^2}}\\
\equiv&\frac{\pi}{2\beta}.
\end{split}
\end{equation}
Finally,  for $p\neq3$ the monotonicity of $L$ in $\tau$ can be easily deduced from \cite[Corollary 5.6]{MR1949167}. The proof is complete.

\proofend

Based on Lemma \ref{31}, we present the proof of Theorem \ref{thm1.1} as follows.\\

\noindent{\bf Proof of Theorem \ref{thm1.1}.}
We  provide  only proofs  for two specific  cases: $p\in(0,1)$ and $p=3.$ This is because the remaining cases,  $1\leq p<3$ and $p>3$ can be proven in the same manner as in the case $p\in(0,1).$
\begin{itemize}
  \item Case $p\in(0,1)$\\
  According to Lemma \ref{31}, the range of the minimum half-period $L$ is $I:=\left(\frac{\pi}{\alpha+2},\frac{\sqrt{p+1}}{\alpha+2}\pi\right).$ Thus, \eqref{2.6} has no nontrivial solution if and only if the interval $I$ does not contain $\frac{\pi}{j}$ for any integer $j\geq 1,$ which implies that either $\alpha\leq-1$ or
\beq \label{3:51}
\alpha\leq j-1 \quad \text{and}\quad p\leq \left(\frac{\alpha+2}{j}\right)^2-1\quad \text{for some }\ j\geq 1.
\eeq
In other words, for $p\in(0,1),$ \eqref{2.6} has no nontrivial solution if and only if $(\alpha,p)\in\mathscr{M}_1$.
On the other hand, for $ (\alpha,p)\in(-2,+\infty)\times(0,1)\setminus\mathscr{M}_1$, in addition to $\mathfrak{S}_0 =\left\{\left(\frac{\lam(p+1)^2}{(\alpha+2)^2}\right)^{\frac{1}{p+1}}\right\}\subset \mathfrak{S}$, $\mathfrak{S}$ also contains nontrivial solutions since the interval $I$ contains $\frac{\pi}{j}$ for some integer $j\geq 1$. In other words, there exists $j\geq1$ such that $\frac{\pi}{\alpha+2}<\frac{\pi}{j}<\frac{\sqrt{p+1}}{\alpha+2}\pi,$ i.e., $\frac{\alpha+2}{\sqrt{p+1}}<j<\alpha+2.$
If we denote  the number of integers in $(\frac{\alpha+2}{\sqrt{p+1}},\alpha+2)$ by $N_0(\alpha,p),$ then $I$ contains
$\{\frac{\pi}{j_i}|j_i=[\frac{\alpha+2}{\sqrt{p+1}}]+i,i=1,2,\cdots,N_0(\alpha,p)\}.$
 This implies that \eqref{2.6} has $N_0(\alpha,p)$ periodic solutions $w_{j_1}(\theta),w_{j_2}(\theta),\cdots,w_{j_{N_0(\alpha,p)}}(\theta),$ where the period of  $w_{j_i}(\theta)$ is $\frac{2\pi}{j_i}$  and $\min\limits_{\theta} w_{j_i}(\theta)=w_{j_i}(0).$
In conclusion, if $ (\alpha,p)\in(-2,+\infty)\times(0,1)\setminus\mathscr{M}_1$, then $\mathfrak{S}$ contains precisely $1+N_0(\alpha,p)$ connected components  $\mathfrak{S}_0=\left\{\left(\frac{\lam}{\beta^2}\right)^{\frac{1}{p+1}}\right\}$, $\mathfrak{S}_1,\cdots ,\mathfrak{S}_i,$ $\cdots ,\mathfrak{S}_{N_0(\alpha,p)}$, where $\mathfrak{S}_i$ is defined by (\ref{thm1.1:2}).

  \item Case $p=3$

  In this case, $L(\tau)\equiv\frac{2}{\alpha+2}\pi.$ Hence, if ${{\alpha+2}\over2}\notin \mathbb{N}^{+},$ i.e., $\alpha\neq0,2,4,6,\cdots,$ then  $\mathfrak{S} =\left\{\frac{2\lam^{\frac{1}{4}}}{\sqrt{\alpha+2}}\right\}.$ Otherwise,
  similar to \cite{MR2796243}, one can show that all solutions of \eqref{2.6} are given by a 2-parameter family of functions \eqref{8}.

\end{itemize}

The proof is complete.
\qed

\section{The  asymptotic behaviors as $|x|\rightarrow 0$ and $|x|\rightarrow \infty$}

This section focuses on  the  asymptotic behavior of the solution $u(r,\theta)$ as $r\rightarrow0$ and $r\rightarrow\infty$. The goal is to prove Theorem \ref{thm1.2}. Note \eqref{in:trans}, it is reduced to study the behaviors of $v(t,\theta)$ as  $t\rightarrow -\infty$ and $t\rightarrow +\infty$.

Because   inequality \eqref{135} in Theorem \ref{thm1.2} can be obtained in a similar manner as in \cite{MR3998250}, we only provide the proof of \eqref{136}. In other words, we   study only the behavior of $v(t,\theta)$ as  $t\rightarrow +\infty$.
We begin with the following  proposition:

\begin{prop}\label{B64}
Let $v$  be a solution of the evolution equation
\begin{equation}\label{B67}
 -v_{tt}- 2\beta v_t=v_{\theta\theta}+ \beta^2 v-\displaystyle\frac{\lam}{v^p}, \quad (t,\theta)\in \{(t_0,+\infty)\times S^1\}
 \end{equation}
and satisfy \eqref{C14}, where $\beta, \lambda, p$ are   positive constants. Then there exists  $ w_{\infty}\in \mathfrak{S}$ such that
\begin{equation}\label{B66}
\begin{split}
\|v(t,\cdot)-w_{\infty}(\cdot)\|_{C^2(S^1)}\leq &C(1+t)^{-\frac{s_2}{7(1-2s_2)}},\quad \text{as}\ t\rightarrow +\infty,
\end{split}
\end{equation}
where $s_2\in (0,\frac{1}{2})$ is a constant  depending on $w_{\infty}$. Here $C$ is a positive constant.
\end{prop}

\begin{proof} In contrast to \cite[Proposition 3.5]{MR3998250},  we need to modify the definition of  functional $\mathscr{H}(v)$.
For any $\varepsilon>0$ and $t\geq t_0$, we define

\beq\label{B41}
\mathscr{H}(v)={1\over2}\int_{S^1}|v_t|^2d\theta-(1+2\beta \varepsilon) E(v)+\varepsilon(v_{\theta\theta}+j(v),v_t)
\eeq
with
\[
E(v)=\int_{S^1}\frac{1}{2}(v_{\theta}^2-g(v))~d\theta,\,\ g(v)=\beta^2v^2+2h(v),
\]
\beq\label{70}
j(v)=\beta^2v-{\lambda\over v^p},
\eeq
\[
(v_{\theta\theta}+j(v),v_t)=\int_{S^1}(v_{\theta\theta}+j(v))v_t~d\theta.
\]
Then, Proposition \ref{B64} can be proved in a similar way as in \cite[Proposition 3.5]{MR3998250} by applying dynamical system theory and the Lojasiewicz-Simon method.\qed
\end{proof}
\noindent\textbf{Proof of Theorem \ref{thm1.2}.}
Inequality \eqref{136} follows directly from Proposition \ref{B64} by \eqref{in:trans}, and   inequality \eqref{135} can be obtained in a similar manner as in \cite{MR3998250}.\qed

\section{Global Solutions}
In this section, we study the global rupture solutions, that is,  the functions  $u\in C^2(\R^2\setminus \{0\})$ that satisfy \eqref{eqn} and \eqref{81}. Our goal is to prove Theorem \ref{121} and \ref{thm1.3}.

Recall  the energy $E$ defined by \eqref{E0} and $(w_0(\theta),w_{\infty}(\theta))$ defined by \eqref{120}. First, we prove that the energy $E$  at the origin is always less than or equal to that at infinity. More precisely, we have

\begin{prop}\label{3.prop1}
It holds that $E(w_0)\leq E(w_{\infty})$.
 Furthermore, if $E(w_0)=E(w_{\infty})$, then $ v(t,\theta)\equiv w_0(\theta)=w_{\infty}(\theta)$ for all $t\in \R$, that is, $u(r,\theta)=r^{\frac{\alp+2}{p+1}}w_0(\theta)$ for all $(r,\theta)\in (0,+\infty)\times S^1$.
\end{prop}

\proof Multiplying \eqref{R:trans-1} by $v_t$ and integrating  over $[-T,T] \times S^1$ yields
\beq \label{C13}
\int_{-T}^{T}\int_{S^1}~\left(-v_{tt}v_t-2\beta{v_t}^2\right) ~d\theta dt=\int_{-T}^{T}\int_{S^1}~\left(v_{\theta\theta}v_t+\beta^2v_tv -\frac{\lambda}{v^p}v_t\right) ~d\theta dt.
\eeq
Note that similar to \cite{MR3998250},  $v_{t}(t,\cdot)$ tends to $0$ in $C^0(S^1)$ as $t\to \pm\infty$. Let $T\to +\infty$, we obtain \beq
\begin{split}
0\leq&\int_{-\infty}^{+\infty}\int_{S^1}2\beta{v_t}^2 ~d\theta dt\\
=&\int_{S^1}\frac{1}{2}\big((w'_{\infty})^2-(w'_0)^2\big)-\frac{\beta^2}{2}\big(w_{\infty}^2-w_0^2\big)-\big(h(w_{\infty})-h(w_0)\big)~d\theta\\
=&E(w_{\infty})-E(w_0)\notag
\end{split}
\eeq
and $E(w_0)\leq E(w_{\infty})$ follows. Moreover,
if $E(w_0)=E(w_{\infty})$, 
 then $v_t(t,\theta)\equiv 0,$ which implies that $v$ is independent of $t$ and $ v(t,\theta)\equiv w_0(\theta)=w_{\infty}(\theta).$
The proof is complete.
\proofend
\begin{cor}\label{82}
If $\mathfrak{S}$ contains only  the trivial solution  $w(\theta)\equiv \big(\frac{\lam}{\beta^2}\big)^{\frac{1}{p+1}},$ then $v(t,\theta)\equiv \big(\frac{\lam}{\beta^2}\big)^{\frac{1}{p+1}}$ is the unique solution of \eqref{R:trans-1} satisfying \eqref{C14}, which means $u(r,\theta)=\big(\frac{\lam}{\beta^2}\big)^{\frac{1}{p+1}}r^{\frac{\alp+2}{p+1}}$ is the unique global solution to \eqref{eqn} satisfying \eqref{81}.
\end{cor}

\begin{cor}\label{83}
By Corollary \ref{82} and Theorem \ref{thm1.1}, if $(\alpha,p)$ satisfies \eqref{thm1.1:1}, then
 $u(r,\theta)=\big(\frac{\lam}{\beta^2}\big)^{\frac{1}{p+1}}r^{\frac{\alp+2}{p+1}}$ is the unique global solution to \eqref{eqn} satisfying \eqref{81}.
\end{cor}

Next,  assume that there exists a function $u(x)=u(r,\theta)$ which is  a global  solution to \eqref{eqn}, satisfying \eqref{81} and \eqref{120} with $w_0(\theta),w_{\infty}(\theta) \in \mathfrak{S}.$ Our goal is to obtain the necessary conditions that the pair of elements $w_0(\theta), w_{\infty}(\theta)$ must satisfy. To do this, note that $E(w_0)\leq E(w_{\infty})$,  we will study how  $E(w)$ depends on $w\in\mathfrak{S}.$
Note that the value of $E(w)$ does not change with the shift of $w.$ Recall that $w_1$ (or $w_2$ ) denotes the minimum (or maximum) value of $w(\theta)$,  $\tau=\frac{w_2}{w_1}$ and for any given $\tau\in[1,+\infty)$, there exists a unique (up to shift) solution $w(\theta)$ to \eqref{2.6} such that $w_2/w_1=\tau.$ We aim to prove the monotonicity of $E(w)$ with respect to $\tau$. In fact, by Lemma \ref{31}, if $p\neq3$,  monotonicity with respect to  $\tau$  means  monotonicity with respect to $L$ (the minimum half-period of $w$); if $p=3$, monotonicity with respect to  $\tau$ means  monotonicity with respect to
$\epsilon$ in Theorem \ref{thm1.1}.

Before continuing, let us give another form of $E(w)$ for $w\in \mathfrak{S}$. Introduce
\beq\label{108}
F(w):=\frac{1}{2\pi}\bigg(\frac{\lambda}{\beta^2}\bigg)^{\frac{p-1}{p+1}}\int_{S^1}~\frac{1}{w^{p-1}} ~d\theta,\ p\neq 1, \quad \text{and} \quad F_1(w):=\frac{1}{2\pi}\int_{S^1}~\ln w ~d\theta.
\eeq
If we  multiply \eqref{2.6} by $w$ and integrate  on $S^1,$ then $E(w)$ on $\mathfrak{S}$ can be rewritten as

\begin{equation}\label{109}
E(w)=\left\{\arraycolsep=1.5pt
\begin {array}{lll}
\int_{S^1}-\left(\frac{1}{2}+\frac{1}{p-1}\right)\frac{\lambda}{w^{p-1}}d\theta=\pi\lambda\frac{1+p}{1-p}  \big(\frac{\lambda}{\beta^2}\big)^{\frac{1-p}{p+1}} F(w), \ &\text{for}\ p\neq1,\\
\int_{S^1}(-\frac{\lambda}{2}+\lambda\ln w )d\theta=-\pi\lambda +2\pi\lambda\ln \sqrt{\frac{\lam}{\beta^2}}+2\pi\lambda F_1(w), \ &\text{for}\  p=1.
\end{array}\right.
\end{equation}
The monotonicity of $E(w)$ with respect to $\tau$ is reduced to obtain the monotonicity of $F(w)$ and $F_1(w)$  with respect to $\tau$. For convenience, we sometimes denote  $E(w)$ as  $E(\tau)$, as  discussed previously.
 The same applies to $F(\tau)$ and $F_1(\tau).$

If we rewrite

\begin{equation}
\begin{split}
F(w)=&\bigg(\frac{2\lambda}{\beta^2}\bigg)^{\frac{p-1}{p+1}}\frac{1}{L(\tau)}\int_{w_1}^{w_2}~\frac{dw}{w^{p-1}\sqrt{E-g(w)}}  ,\\
F_1(w)=&\frac{1}{L(\tau)}\int_{w_1}^{w_2}~\frac{\ln w dw}{\sqrt{E-g(w)}} ,
\end{split}
\end{equation}
by \eqref{90},    it should be noted that  both $F$ and $F_1$ depend only on $\tau$ and $p,$ not on $\beta$ and $\lambda.$ In fact,
as in \eqref{14} and \eqref{112}, by setting
$y=\frac{w}{w_1}$ and $\xi=\frac{y-1}{\tau-1}$,  $F(w)$ and $F_1(w)$ can be rewritten   as
\begin{equation}\label{111}
F(\tau)=  \frac{\displaystyle{\int_0^1~\left(\frac{2(1-\tau^{1-p})}{(p-1)(\tau ^2-1)}\right)^{\frac{1-p}{p+1}}
\frac{(\xi(\tau-1)+1)^{1-p}(\tau-1)}{\sqrt{1+\frac{\tau^{p-1}(\tau^2-1)}{\tau^{p-1}-1}-(\xi(\tau-1)+1)^2-\frac{\tau^{p-1}(\tau^2-1)}{\tau^{p-1}-1}{1\over{(\xi(\tau-1)+1)^{p-1}}}}}~d\xi}}
{\displaystyle{\int_0^1\frac{\tau-1}{\sqrt{1+\frac{\tau^{p-1}(\tau^2-1)}{\tau^{p-1}-1}-(\xi(\tau-1)+1)^2-\frac{\tau^{p-1}(\tau^2-1)}{\tau^{p-1}-1}{1\over{(\xi(\tau-1)+1)^{p-1}}}}}~d\xi}}
\end{equation}
and
\begin{equation}\label{113}
F_1(\tau)= \frac{\displaystyle{\int_0^1 ~\ln \left[\left(\frac{2 \ln\tau}{\tau ^2-1}\right)^{\frac{1}{2}}(\xi(\tau-1)+1)\right] \frac{\tau-1}{\sqrt{1-(\xi(\tau-1)+1)^2+(\tau^2-1)\frac{\ln{(\xi(\tau-1)+1)}}{\ln \tau}}}~d\xi}}
{\displaystyle{\int_0^1~\frac{\tau-1}{\sqrt{1-(\xi(\tau-1)+1)^2+(\tau^2-1)\frac{\ln{(\xi(\tau-1)+1)}}{\ln \tau}}}~d\xi}}.
\end{equation}

In the remainder  of this section, for convenience, we denote the numerator in  formula \eqref{111} as
\begin{equation}\label{3.H}
\begin{split}
 H(\tau):=&\int_0^1 ~\left(\frac{2(1-\tau^{1-p})}{(p-1)(\tau ^2-1)}\right)^{\frac{1-p}{p+1}}
\frac{(\xi(\tau-1)+1)^{1-p}(\tau-1)}{\sqrt{1+\frac{\tau^{p-1}(\tau^2-1)}{\tau^{p-1}-1}-(\xi(\tau-1)+1)^2-\frac{\tau^{p-1}(\tau^2-1)}{\tau^{p-1}-1}{1\over{(\xi(\tau-1)+1)^{p-1}}}}}~\mathrm{d}\xi\\
\overset{\triangle}{=}&\int_0^1 G(\xi,\tau,p)\mathrm{d}\xi.
\end{split}
\end{equation}
Then by \eqref{112}, $F(\tau)$ can also be rewritten as
\beq\label{1010}
F(\tau)=\frac{H(\tau)}{\beta L(\tau)}.
\eeq

In the following, the  cases  $p=3,$ $p\in(0,+\infty)\setminus\{1,3\}$ and $p=1$ are considered separately because of their differences. More precisely, in Section \ref{115}, we  prove that   $E(w)$ is a constant on $\mathfrak{S}$ for a particular case $p=3.$
For cases $p\in(0,+\infty)\setminus\{1,3\}$ and $p=1$,
 we   study the asymptotic behavior of  $E(w)$  with respect to $\tau$ in Sections \ref{117} and \ref{127}, respectively.
In Section \ref{134}, we discuss the monotonicity of $E(\tau)$ with respect to $\tau$.
 Finally, we provide the proofs of Theorems \ref{121} and \ref{thm1.3} in Section \ref{128}.

\subsection{The particular case $p=3$}\label{115}
In this subsection, we show that, for the particular case $p=3$, $F(\tau)$ is a constant independent of $\tau$, which implies  by \eqref{109} that $E(w)$ is a constant on $\mathfrak{S}$. More precisely, we have
\begin{prop}\label{110} For  $p=3$,  $F(\tau)\equiv 1$ holds for all $\tau\in [1,+\infty),$  which implies  by \eqref{109} that $E(w)\equiv -\pi\sqrt{\lambda}(\alpha+2)/2$ on $ \mathfrak{S}$.
\end{prop}
\proof{
Using \eqref{111} and \eqref{112}, and recalling that $L(\tau)\equiv\frac{\pi}{2\beta}$ for $p=3,$
by \eqref{1010} we rewrite
$F(\tau)=\frac{H(\tau)}{\beta L(\tau)}=\frac{2}{\pi}H(\tau).$ Therefore, it is reduced to prove $H(\tau)\equiv\pi/2.$ In fact,
\begin{equation}
\begin{split}
H(\tau)=&\int_0^1~\left(\frac{2(1-\tau ^{-2})}{2(\tau ^2-1)}\right)^{-\frac{1}{2}}\frac{(\xi(\tau-1)+1)^{-2}(\tau-1)}{\sqrt{1+\tau^2-(\xi(\tau-1)+1)^2-\frac{\tau^2}{(\xi(\tau-1)+1)^2}}}~d\xi\\
=&\tau\int_0^1~\frac{(\tau-1)(\xi(\tau-1)+1)^{-1}}{\sqrt{(1+\tau^2)(\xi(\tau-1)+1)^2-(\xi(\tau-1)+1)^4-\tau^2}}~d\xi\\
\xlongequal{x=(\xi(\tau-1)+1)^2}&\frac{\tau}{2}\int_1^{\tau^2}\frac{dx}{x\sqrt{(\frac{\tau^2-1}{2})^2-(x-\frac{\tau^2+1}{2})^2}}\\
\xlongequal{x-\frac{\tau^2+1}{2}=\frac{\tau^2-1}{2}\sin X}&\frac{\tau}{2}\int_{-\frac{\pi}{2}}^{\frac{\pi}{2}}\frac{dX}{\frac{\tau^2+1}{2}+\frac{\tau^2-1}{2}\sin X}\\
\xlongequal{Y=\tan\frac{X}{2}}&\frac{\tau^2+1}{2\tau}\int_{-1}^{1}\frac{dY}{1+
\left({\frac{\tau^2+1}{2\tau}}(Y+\frac{\tau^2-1}{\tau^2+1})\right)^2}\\
=&[\arctan{\tau}-\arctan(-\frac{1}{\tau})]\\
\equiv &\pi/2.
\end{split}
\end{equation}
So the proof  is complete.
 \proofend
}

If we combine Proposition \ref{3.prop1}, Proposition \ref{110} and Theorem \ref{thm1.1}, we obtain
\begin{cor}\label{0805}
For $p=3,$  $u(r,\theta)$ is a  global solution to \eqref{eqn} satisfying \eqref{81} if and only if  there exists $w_{\epsilon,a}(\theta)\in \mathfrak{S}$ such that $u(r,\theta)=w_{\epsilon,a}(\theta)r^{\frac{\alp+2}{p+1}}.$ In particular, for $p=3$ and $\alp \notin 2\N,$  $u(r,\theta)=\big(\frac{\lam}{\beta^2}\big)^{\frac{1}{p+1}}r^{\frac{\alp+2}{p+1}}$ is the unique global solution to \eqref{eqn} satisfying \eqref{81}.
\end{cor}

\subsection{The asymptotic behavior of $E(\tau)$ for $p\in(0,+\infty)\setminus\{1,3\}$}\label{117}
Unlike the case $p=3,$ for cases $p\neq 3$, it is challenging to confirm whether $E(\tau)$ is monotonous with respect to $\tau$, even though we believe it is.
To observe   monotonicity,  we first study the asymptotic behavior of $E(\tau)$ as $\tau \to 1$ and $\tau \to +\infty.$ Using \eqref{109}, this is reduced to the study of   asymptotic behavior of $F(\tau)$ for $p\neq1$ and $F_1(\tau)$ for $p=1.$ In this subsection, we   focus on the asymptotic behavior of $F(\tau)$, and the asymptotic behavior of $F_1(\tau)$ is considered in subsection \ref{127}.

\subsubsection{The asymptotic behavior of $E(\tau)$ as $\tau \to 1$}

\begin{prop}\label{prop3.2} It holds $\lim_{\tau\to 1}F(\tau)=1$ for all $p>0$ and $p\neq1,$ which implies $\lim_{\tau\to 1}E(\tau)=\pi\lambda\frac{1+p}{1-p}  \big(\frac{\lambda}{\beta^2}\big)^{\frac{1-p}{p+1}}$ by \eqref{109}.
\end{prop}

\noindent{\bf Proof:} Note \eqref{1010},
  and the asymptotic behavior  of $L(\tau)$ is given in Lemma \ref{31}. Therefore, it is reduced  to  study  the behavior of
$H(\tau)$ as $\tau \to 1.$ 
Denote
\begin{equation}
\begin{split}
f_1(\xi,\tau):=&(\xi(\tau-1)+1)^{1-p}\left(\frac{2(1-\tau^{1-p})}{(p-1)(\tau ^2-1)}\right)^{\frac{1-p}{p+1}},\\
f_2(\xi,\tau):=&\frac{\tau-1}{\sqrt{1+\frac{\tau^{p-1}(\tau^2-1)}{\tau^{p-1}-1}-(\xi(\tau-1)+1)^2-\frac{\tau^{p-1}(\tau^2-1)}{\tau^{p-1}-1}{1\over{(\xi(\tau-1)+1)^{p-1}}}}}.
\end{split}
\end{equation}
Then $H(\tau)=\int_{0}^{1}~f_1(\xi,\tau)f_2(\xi,\tau)~d\xi$,\,\ $\beta L(\tau)=\int_{0}^{1}~f_2(\xi,\tau)~d\xi$.
We claim that  $H(\tau)=\int_{0}^{1}~f_1(\xi,\tau)f_2(\xi,\tau)~d\xi$ converges uniformly with respect to $\tau\in[1,2]$.
In fact, for $p>1$, note that the  improper integral $L(\tau)$ converges uniformly on $[1,2]$, $f_1(\xi,\tau)$  monotonically decreases in $\xi$ and is uniformly bounded  on $[1,2]$. The uniform convergence of  $H(\tau)$  follows by the Abel test.
For $p<1$, note that
$$H(\tau)\leq\tau^{1-p}\left(\frac{2(1-\tau^{1-p})}{(p-1)(\tau ^2-1)}\right)^{\frac{1-p}{p+1}}\beta L(\tau),$$ the uniform convergence of $H(\tau)$ on $[1,2]$ also holds.

Thus, by \eqref{1010} and $\lim_{\tau\to 1}~f_1(\xi,\tau)=1$ we have

\begin{equation}
\begin{split}
\lim_{\tau\to 1}F(\tau)=\lim_{\tau\to 1}\frac{\int_0^1 ~f_1(\xi,\tau)f_2(\xi,\tau)~d\xi}{\int_0^1 ~f_2(\xi,\tau)~d\xi}
=\frac{\int_0^1\lim_{\tau\to 1}~f_1(\xi,\tau) f_2(\xi,\tau)~d\xi}{\int_0^1\lim_{\tau\to 1}~f_2(\xi,\tau)~d\xi}\notag
=1.
\end{split}
\end{equation}
The proof  is complete.
\qed

\subsubsection{The asymptotic behavior of $E(\tau)$ as $\tau\to \infty$} \label{126}

The cases $p\in(0,1),$ $p\in(3,+\infty)$ and $p\in(1,3)$ are considered in propositions \ref{prop3.1}, \ref{122}, and \ref{123}, respectively.  Note that the study of   case $p\in(1,3)$ is  more complicated than that of the other two cases.  In fact, these involve elliptic integrals.

\begin{prop}\label{prop3.1} Let $p\in(0,1).$ Suppose that $H(\tau)$ is defined in \eqref{3.H}. Then we have
\beq
\lim_{\tau\to \infty}H(\tau) \overset{\triangle}{=}H(+\infty)=\left(\frac{2}{1-p}\right)^{\frac{1-p}{p+1}}\int_0^1\frac{\xi^{1-p}}{\sqrt{\xi^{1-p}-\xi^2}}d\xi.
\eeq
By Lemma \ref{31} and \eqref{109} this  implies that
\beq
\lim_{\tau\to \infty}F(\tau) \overset{\triangle}{=}C(p)=\frac{p+1}{\pi}H(+\infty)
\eeq
and
\beq
\lim_{\tau\to \infty}E(\tau) =\pi\lambda\frac{1+p}{1-p}  \big(\frac{\lambda}{\beta^2}\big)^{\frac{1-p}{p+1}} C(p).
\eeq
\end{prop}
\begin{proof}
As in the proof of Proposition \ref{prop3.2}, it can be seen that $H(\tau)$ converges uniformly on $[1,+\infty)$. Hence, we have
$$\lim_{\tau\to \infty}H(\tau)=\int_0^1 \lim_{\tau\to \infty} G(\xi,\tau,p)\mathrm{d}\xi=\left(\frac{2}{1-p}\right)^{\frac{1-p}{p+1}}\int_0^1\frac{\xi^{1-p}}{\sqrt{\xi^{1-p}-\xi^2}}d\xi$$
and the proof is complete.
\qed
\end{proof}

\begin{rem}
  We point out that $F(+\infty)>1$ for all $p\in(0,1).$
\end{rem}

Next, in preparation for what follows,  rewrite  $H(\tau)$ in \eqref{3.H} as
\beq\label{41}
H(\tau)=A(\tau)\int_1^{\tau}\frac{dy}{y^{p-1}\sqrt{I(y,\tau)}}
\eeq
with $A(\tau)=\left(\frac{2(1-\tau^{1-p})}{(p-1)(\tau ^2-1)}\right)^{\frac{1-p}{p+1}},$ $I(y,\tau)=1+z(\tau)-y^2-z(\tau)y^{1-p}$ and $z(\tau)=\frac{\tau^{p-1}(\tau^2-1)}{\tau^{p-1}-1}.$

For the case $p>3,$ the asymptotic behavior of $E(\tau)$ as $\tau\to \infty$ can be obtained as follows:

\begin{prop}\label{122} Let $p>3.$  Then it holds
\beq
\lim_{\tau\to \infty}F(\tau) =+\infty,
\eeq
which implies $\lim_{\tau\to \infty}E(\tau) =+\infty.$
\end{prop}
\begin{proof}
Note that $F(\tau)=H(\tau)/\beta L(\tau)$ and $\lim_{\tau\rightarrow +\infty}L(\tau)=\pi/2$ by Lemma \ref{31}. It is reduced to proving
\beq\label{42}
\lim_{\tau\to \infty}H(\tau) =+\infty.
\eeq
Recall \eqref{41}. By noting that
\beq
I(y,\tau)\leq I_y(1,\tau)(y-1), \quad \forall \ y\in(1,\tau)
\eeq
due to the concavity of $I$ with respect to $y$ and $I(1,\tau)=0$. It  holds that
\beq
\begin{split}
H(\tau)=&A(\tau)\int_1^{\tau}\frac{dy}{y^{p-1}\sqrt{I(y,\tau)}}\\
\geq& A(\tau)\int_1^2\frac{dy}{y^{p-1}\sqrt{I_y(1,\tau)(y-1)}}\\
\geq& A(\tau)\frac{1}{2^{p-1}\sqrt{-2+z(\tau)(p-1)}}\int_1^2\frac{dy}{\sqrt{y-1}}\\
\geq & C(p) \tau^{\frac{p-3}{p+1}}
\end{split}
\eeq
for $\tau$ large, where $C(p)$ is a positive constant that depends only on $p.$ Therefore, \eqref{42} holds and  the proof is complete.
\qed
\end{proof}

In the remainder of this subsection, we   focus on   studying   the asymptotic behavior of $E(\tau)$ as $\tau\rightarrow +\infty$ for the case $1<p<3,$ which is more complicated than the previous two cases.  In fact, these involve elliptic integrals. In preparation, we provide the following lemmas:

\begin{lem}\label{43}
It holds
 \beq
 K(k^2)\overset{\triangle}{=}\int_0^{\pi\over2}\frac{d\theta}{\sqrt{1-k^2sin^2\theta}}
 =\frac{{\pi\over2}}{AGM(1,\sqrt{1-k^2})}.
 \eeq
where $K(k^2)$ is  the complete elliptic integral of the first kind and $AGM(x,y)$ is the arithmetic-geometric mean of two positive real numbers $x$ and $y.$
\end{lem}
This  is a well-known result, and from which we  obtain
\begin{lem}\label{44}
It holds
\beq
\int_1^{\tau}\frac{1}{\sqrt{y}\sqrt{y-1}\sqrt{\tau-y}}dy=\frac{2}{\sqrt{\tau}}K\left(1-{1\over\tau}\right)=\frac{2}{\sqrt{\tau}}\frac{{\pi/2}}{AGM({1\over{\sqrt{\tau}}},1)}.
\eeq
\end{lem}
\begin{proof}
By calculation, we have
\beq
\begin{split}
&\int_1^{\tau}\frac{1}{\sqrt{y}\sqrt{y-1}\sqrt{\tau-y}}dy\\
=&\int_1^{\tau}\frac{1}{\sqrt{y}\sqrt{\frac{(\tau-1)^2}{4}-\left(y-\frac{\tau+1}{2}\right)^2}}dy\\
\xlongequal{y-\frac{\tau+1}{2}=\frac{\tau-1}{2}\sin X}=&\int_{-\pi/2}^{\pi/2}\frac{dX}{\sqrt{\frac{\tau+1}{2}+\frac{\tau-1}{2}\sin X}}\\
=&\int_{0}^{\pi}\frac{dX}{\sqrt{\frac{\tau+1}{2}+\frac{\tau-1}{2}\cos X}}\\
\xlongequal{X=2\theta}&\frac{2}{\sqrt{\tau}}\int_{0}^{\pi/2}\frac{d\theta}{\sqrt{1-(1-\frac{1}{\tau})\sin^2\theta }}\\
=&\frac{2}{\sqrt{\tau}}K\left(1-{1\over\tau}\right).
\end{split}
\eeq
Thus the proof is complete by Lemma \ref{43}.
\qed
\end{proof}

Second, we present another  result (see,e.g. \cite[Theorem 2.5]{MR1010408}) as follows:
\begin{lem}\label{45}
$\frac{{\pi/2}}{AGM(k,1)}\sim \ln(\frac{4}{k})$ as $k\rightarrow 0^+.$

\end{lem}

Based on  these preparations, we can derive the asymptotic behavior of $F(\tau)$ as $\tau\rightarrow +\infty$ for $1<p<3.$
\begin{prop} \label{123} Let $1<p<3.$ Then it holds
\beq
\lim_{\tau\to \infty}E(\tau)=\lim_{\tau\to \infty}F(\tau) =0.
\eeq
\end{prop}
\begin{proof}
Note \eqref{109}, $F(\tau)=H(\tau)/\beta L(\tau)$ and $\lim_{\tau\rightarrow +\infty}L(\tau)=\pi/2$ by Lemma \ref{31}. It is reduced to prove
\beq\label{46}
\lim_{\tau\to \infty}H(\tau) =0\quad \text{for}\quad p\in(1,3).
\eeq
By calculation, $I(y,\tau)$ in \eqref{41} can be rewritten as
\beq
I(y,\tau)=\frac{z(\tau)-z(y)}{\tau-y}\frac{y^{p-1}-1}{y^{p-1}(y-1)}(y-1)(\tau-y)
\eeq
and  we have
\beq
I(y,\tau)\geq\frac{z(\tau)-z(1)}{\tau-1}\frac{y^{p-1}-1}{y^{p-1}(y-1)}(y-1)(\tau-y)\quad \text{for}\ y\in(1,\tau),
\eeq
by using the convexity of $z(y)$ with respect to $y.$ Then from \eqref{41},  Lemmas \ref{44}, and \ref{45}, we have
\beq
\begin{split}
H(\tau)&\leq A(\tau)\left(\frac{z(\tau)-z(1)}{\tau-1}\right)^{-1/2}\int_1^{\tau}\left(y^{p-1}\sqrt{\frac{y^{p-1}-1}{y^{p-1}(y-1)}}\sqrt{y-1}\sqrt{\tau-y}\right)^{-1}dy\\
&\leq A(\tau)\left(\frac{z(\tau)-z(1)}{\tau-1}\right)^{-1/2}\int_1^{\tau}\left(y^{p-\frac{3}{2}}\sqrt{\frac{y^{p-1}-1}{y^{p-1}(y-1)}}\sqrt{y}\sqrt{y-1}\sqrt{\tau-y}\right)^{-1}dy\\
&\leq C(p) \tau^{\frac{2(p-1)}{p+1}-{1\over2}}\int_1^{\tau}y^{2-p}\left(\sqrt{y}\sqrt{y-1}\sqrt{\tau-y}\right)^{-1}dy\\
&\leq \left\{\arraycolsep=1.5pt
\begin {array}{lll}
C(p) \tau^{\frac{2(p-1)}{p+1}-{1\over2}}\tau^{2-p}\int_1^{\tau}\left(\sqrt{y}\sqrt{y-1}\sqrt{\tau-y}\right)^{-1}dy, \ &\text{for}\ p\in(1,2)\\
C(p) \tau^{\frac{2(p-1)}{p+1}-{1\over2}}\int_1^{\tau}\left(\sqrt{y}\sqrt{y-1}\sqrt{\tau-y}\right)^{-1}dy,  \ &\text{for}\ p\in [2,+\infty)\\
\end{array}\right.\\
&=\left\{\arraycolsep=1.5pt
\begin {array}{lll}
C(p) \tau^{\frac{2(p-1)}{p+1}-{1\over2}}\tau^{2-p}\frac{2}{\sqrt{\tau}}\frac{{\pi/2}}{AGM({1\over{\sqrt{\tau}}},1)}, \ &\text{for}\ p\in(1,2)\\
C(p) \tau^{\frac{2(p-1)}{p+1}-{1\over2}}\frac{2}{\sqrt{\tau}}\frac{{\pi/2}}{AGM({1\over{\sqrt{\tau}}},1)},  \ &\text{for}\ p\in [2,+\infty)\\
\end{array}\right.\\
&\leq \left\{\arraycolsep=1.5pt
\begin {array}{lll}
C(p) \tau^{\frac{-(p-1)^2}{p+1}}\ln(4\sqrt{\tau}), \ &\text{for}\ p\in(1,2)\\
C(p) \tau^{-\frac{3-p}{p+1}}\ln(4\sqrt{\tau}),  \ &\text{for}\ p\in [2,+\infty),\\
\end{array}\right.\\
\end{split}
\eeq
which implies \eqref{46}. The proof is complete.
\qed
\end{proof}
\begin{rem}
  Finally, we summarize  the results obtained in  this subsection \ref{126} as follows:
  \beq
  \lim_{\tau\to \infty}F(\tau) =\left\{\arraycolsep=1.5pt
  \begin {array}{lll}
C(p)<+\infty, \ &\ 0<p<1,\\
0,  \ &\ 1<p<3,\\
+\infty, \ &\ p>3.\\
\end{array}\right.\\
  \eeq
\end{rem}

\subsection{The asymptotic behavior of $E(\tau)$ for $p=1$}\label{127}
In this subsection, we will focus on the asymptotic behavior of $F_1(\tau).$
\begin{prop}\label{prop1025} Let $p=1$. Then, it holds $\lim_{\tau\to 1}F_1(\tau)=0$, which implies by \eqref{109} that $\lim_{\tau\to 1}E(\tau)=-\pi\lambda +2\pi\lambda\ln \frac{2\sqrt{\lam}}{\alpha+2}$.
\end{prop}
\begin{proof}
Note that
\beq
\lim_{\tau\to 1}\ln \left[\left(\frac{2 \ln\tau}{\tau ^2-1}\right)^{\frac{1}{2}}(\xi(\tau-1)+1)\right]=0,\ \forall \xi\in(0,1).
\eeq
This proposition can be proven in the same manner as   Proposition \ref{prop3.2}.
\qed
\end{proof}

\begin{prop}\label{prop0727}Let $p=1$.  Then, it  holds $\lim_{\tau\to +\infty}F_1(\tau)=+\infty$, which implies by \eqref{109} that $\lim_{\tau\to 1}E(\tau)=+\infty$. More precisely, $F_1(\tau)=\ln \sqrt{\ln \tau}+o(\ln \tau)$ as $\tau\to +\infty.$
\end{prop}
\begin{proof}
Rewrite $F_1$ in \eqref{113} as $$F_1(\tau)=I_1(\tau)+I_2(\tau),$$
where $I_1(\tau)=\frac{1}{2}\ln\frac{2\ln\tau}{\tau^2-1}$ and
\beq
I_2(\tau)= \frac{\displaystyle{\int_0^1 ~\ln \left[(\xi(\tau-1)+1)\right] \frac{\tau-1}{\sqrt{1-(\xi(\tau-1)+1)^2+(\tau^2-1)\frac{\ln{(\xi(\tau-1)+1)}}{\ln \tau}}}~d\xi}}
{\displaystyle{\int_0^1~\frac{\tau-1}{\sqrt{1-(\xi(\tau-1)+1)^2+(\tau^2-1)\frac{\ln{(\xi(\tau-1)+1)}}{\ln \tau}}}~d\xi}}.
\eeq
 Note that $I_2(\tau)=\ln \tau +o(\ln \tau)$ as $\tau\to \infty$ by Abel's Lemma and $$\lim_{\tau\to \infty}\frac{\ln \left[(\xi(\tau-1)+1)\right] }{\ln \tau}=1$$ for any given $\xi\in(0,1)$. Therefore, we can conclude that $F_1(\tau)=\ln \sqrt{\ln \tau}+o(\ln \tau),$ and this proposition follows.

\qed
\end{proof}

\subsection{The monotonicity of $E(\tau)$}\label{134}
In this subsection, we examine the  monotonicity of $E(\tau)$ with respect to $\tau.$ It was previously established in Subsection \ref{115} that $E(\tau)$ is a constant independent of $\tau$ for the particular case $p=3.$  However, for $p\in(0,+\infty)\setminus\{3\}$,  the monotonicity remains uncertain.

  As an initial step towards studying the monotonicity of  $E(\tau),$ we reduce it to examine the monotonicity of $F(\tau)$ and $F_1(\tau),$ as per \eqref{109}. For   easy reference, Table \ref{130} provides a summary of the asymptotic behaviors of  $F(\tau)$ and $F_1(\tau)$  obtained in Subsections \ref{115}-\ref{127}.
From Table  \ref{130}, we can infer that   if  $F(\tau)$ and $F_1(\tau)$  are monotonic with respect to $\tau$, then they can only be expressed as follows:\\
  \beq \label{131}
  \begin{split}
 & (1)\  F(\tau)\ \text{decreases monotonically  with respect to}\ \tau\ \text{for}\  p\in(1,3);\\
 & (2)\  F(\tau)\ \text{increases monotonically  with respect to}\ \tau\ \text{for}\ p\in(0,1)\cup(3,+\infty);\\
 & (3)\   F_1(\tau)\ \text{increases monotonically  with respect to}\ \tau\ \text{for}\ p=1.
  \end{split}
    \eeq
However, providing a mathematical proof for \eqref{131} is challenging. Therefore, we present the numerical results (Figure \ref{level}).

\begin{table}[!h]
\centering
\begin{tabular}{c|c|c} 
\hline
\centering
$F(\tau)$  &  $\tau\to 1$ &  $\tau\to \infty$  \\ 
\hline 
$p<1$ & $1$ & $C(p)(>1)$ \\
\hline
$1<p<3$ & $1$ & $0$  \\
\hline
$p=3$ & $1$ & $1$ \\
\hline
$p>3$ & $1$ & $+\infty$ \\
\toprule

\hline
$F_1(\tau)$  &  $\tau\to 1$ &  $\tau\to \infty$  \\ 
\hline 
$p=1$ & $0$ & $+\infty$ \\
\toprule

\hline
$E(\tau)$  &  $\tau\to 1$ &  $\tau\to \infty$  \\ 
\hline 
$p<1$ & $\pi\lambda\frac{1+p}{1-p}  \big[\frac{(p+1)^2}{(\alpha+2)^2}\lambda\big]^{\frac{1-p}{p+1}}$ & $\pi\lambda\frac{1+p}{1-p}  \big(\frac{\lambda}{\beta^2}\big)^{\frac{1-p}{p+1}} C(p)$ \\
\hline 
$p=1$ & $-\pi\lambda +2\pi\lambda\ln \frac{2\sqrt{\lam}}{\alpha+2}$ & $+\infty$ \\
\hline
$1<p<3$ &$\pi\lambda\frac{1+p}{1-p}  \big[\frac{(p+1)^2}{(\alpha+2)^2}\lambda\big]^{\frac{1-p}{p+1}}$ & $0$  \\
\hline
$p=3$ & $-\pi\sqrt{\lambda}(\alpha+2)/2$ & $-\pi\sqrt{\lambda}(\alpha+2)/2$ \\
\hline

$p>3$ & $\pi\lambda\frac{1+p}{1-p}  \big[\frac{(p+1)^2}{(\alpha+2)^2}\lambda\big]^{\frac{1-p}{p+1}}$ & $+\infty$ \\
\toprule
\end{tabular}
\caption{Asymptotic behaviors of $F(\tau),$ $F_1(\tau)$ and  $E(\tau)$} \label{130} 
\end{table}

\begin{figure}[!h]
	\centering  
	\vspace{-0.35cm} 
	\subfigtopskip=2pt 
	\subfigbottomskip=2pt 
	\subfigcapskip=-5pt 
	\subfigure[$p\in(0,1)\cup(1,3.5)$]{
		\label{level.sub3.1.1}
		\includegraphics[width=0.4\linewidth]{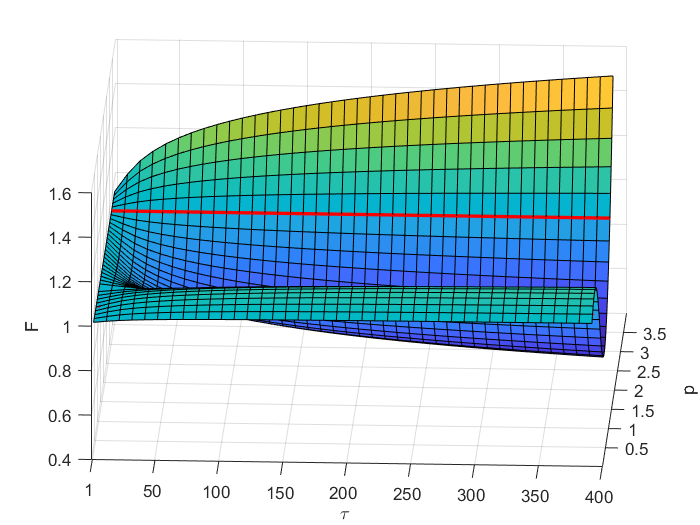}}
	\quad 
    \subfigure[$p\in(3,50)$]{
		\label{level.sub3.1.2}
		\includegraphics[width=0.4\linewidth]{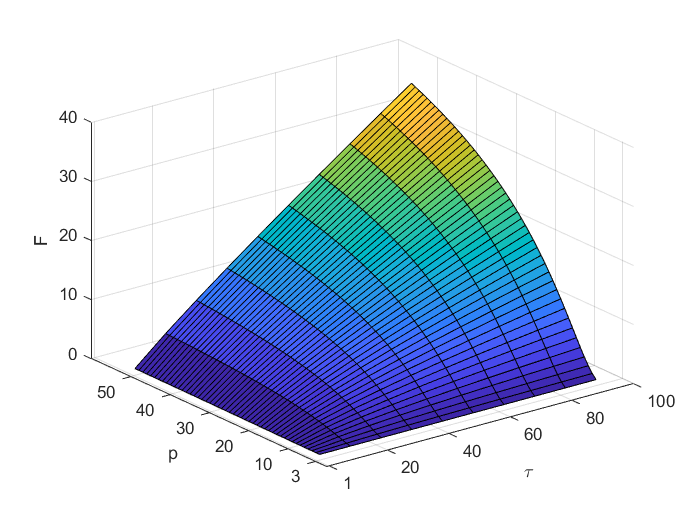}}\\

	\subfigure[$p=0.5,2,3,4$]{
		\label{level.sub3.1.3}
		\includegraphics[width=0.31\linewidth]{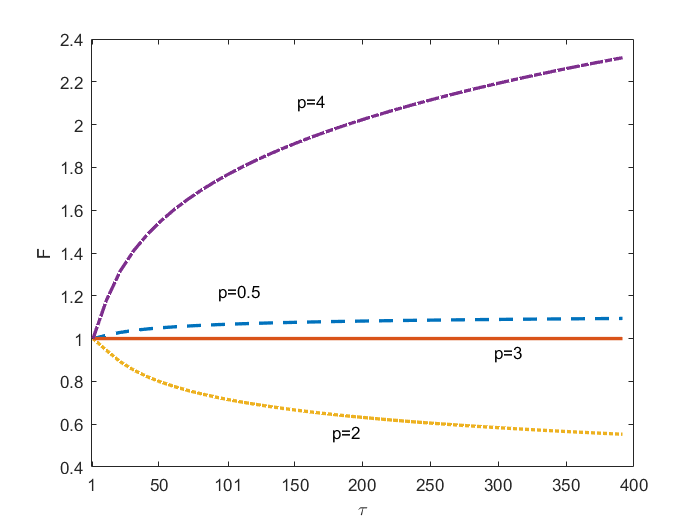}}
	\subfigure[$p=20$]{
		\label{level.sub3.1.4}
		\includegraphics[width=0.31\linewidth]{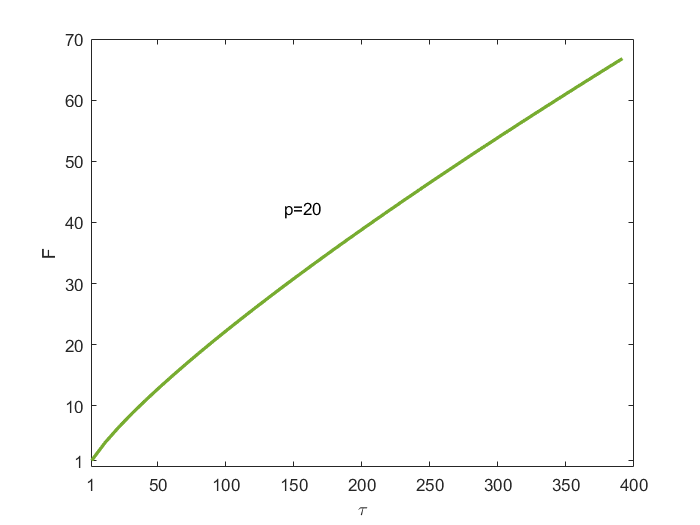}}
 \subfigure[$p=1$]{
		\label{level.sub3.2..1}
		\includegraphics[width=0.31\linewidth]{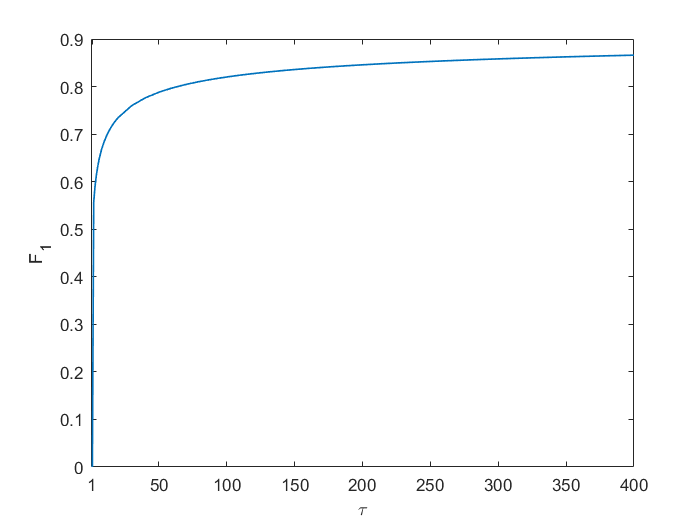}}
%

	\caption{$F$ and $F_1$}
	\label{level}
\end{figure}

\subsection{The proof of Theorems \ref{121} and \ref{thm1.3}}\label{128}
\noindent\textbf{Proof of Theorem \ref{121}.} This theorem can be  obtained directly from Theorem \ref{thm1.1}, Corollary \ref{83} and \ref{0805}.

\noindent\textbf{Proof of Theorem \ref{thm1.3}.}  Combining \eqref{131} with Lemma \ref{31} and \eqref{109}, and recalling that $w\in\mathfrak{S}_0\Leftrightarrow \tau=1,$ leads to the presentation of Table \ref{132}.
\begin{table}[!h]

\begin{tabular}{p{2cm}<{\centering}|p{1.5cm}<{\centering}|p{1.5cm}<{\centering}|p{1.5cm}<{\centering}|p{1.5cm}<{\centering}|p{1.5cm}<{\centering}|p{2cm}<{\centering}} 

\hline
\centering
\  & $j=\frac{\pi}{L(\tau)}$ & $L(\tau)$& $\tau$   & $F(\tau)$ & $E(\tau)$ & $E(\mathfrak{S}_0)$ \\ 
\hline 
$p<1$ & $\nearrow$ & $\searrow$& $\nearrow$   & $\nearrow$ & $\nearrow$ & minimum \\
\hline
$p=1$ & $\nearrow$ & $\searrow$& $\nearrow$   & $\nearrow$ & $\nearrow$ & minimum \\
\hline
$1<p<3$ & $\nearrow$ & $\searrow$& $\nearrow$   & $\searrow$ & $\nearrow$ & minimum \\
\hline
$p>3$ & $\nearrow$ & $\searrow$ & $\searrow$   & $\searrow$ & $\nearrow$ & maximum \\
\hline
\end{tabular}
\centering
\caption{Monotonicity of $E(w)$ with respect to the frequency $j$ of $w$} \label{132}
\end{table}

By Table \ref{132}, we can conclude the subsequent energy inequalities:
\begin{equation}\label{133}
\begin {array}{lll}
  E(\mathfrak{S}_0)<E(\mathfrak{S}_1)<E(\mathfrak{S}_2)<\cdots <E(\mathfrak{S}_j)<\cdots <E(\mathfrak{S}_{N_0}),\quad \text{for}\ 0<p<3,\\
  E(\mathfrak{S}_1)<E(\mathfrak{S}_2)<\cdots <E(\mathfrak{S}_j)<\cdots <E(\mathfrak{S}_{N_0})<E(\mathfrak{S}_0),\quad \text{for}\ p>3,
  \end{array}
  \end{equation}
where $\mathfrak{S}_k,k=1,2,\ldots,N_0$ are as stated in Theorem \ref{thm1.1}. Recall that  $E(w_0)\leq E(w_{\infty})$ from Proposition \ref{3.prop1}.  Theorem \ref{thm1.3} follows and the proof is complete.


\begin{thebibliography}{10}

\bibitem{MR1949167}
Ben Andrews.
\newblock Classification of limiting shapes for isotropic curve flows.
\newblock {\em J. Amer. Math. Soc.}, 16(2):443--459, 2003.

\bibitem{MR1010408}
J.~M. Borwein and P.~B. Borwein.
\newblock A cubic counterpart of {J}acobi's identity and the {AGM}.
\newblock {\em Trans. Amer. Math. Soc.}, 323(2):691--701, 1991.

\bibitem{MR0993442}
Xu-Yan Chen, Hiroshi Matano, and Laurent V\'{e}ron.
\newblock Anisotropic singularities of solutions of nonlinear elliptic
  equations in {${\bf R}^2$}.
\newblock {\em J. Funct. Anal.}, 83(1):50--97, 1989.

\bibitem{MR4631976}
Rodrigo Clemente, Jo\~{a}o~Marcos do~\'{O}, Esteban da~Silva, and Evelina
  Shamarova.
\newblock Touchdown solutions in general {MEMS} models.
\newblock {\em Adv. Nonlinear Anal.}, 12(1):Paper No. 20230102, 18, 2023.

\bibitem{MR3436432}
Juan D\'{a}vila, Kelei Wang, and Juncheng Wei.
\newblock Qualitative analysis of rupture solutions for a {MEMS} problem.
\newblock {\em Ann. Inst. H. Poincar\'{e} C Anal. Non Lin\'{e}aire},
  33(1):221--242, 2016.

\bibitem{MR2957549}
Juan D\'{a}vila and Juncheng Wei.
\newblock Point ruptures for a {MEMS} equation with fringing field.
\newblock {\em Comm. Partial Differential Equations}, 37(8):1462--1493, 2012.

\bibitem{MR3023006}
Juan D\'{a}vila and Dong Ye.
\newblock On finite {M}orse index solutions of two equations with negative
  exponent.
\newblock {\em Proc. Roy. Soc. Edinburgh Sect. A}, 143(1):121--128, 2013.

\bibitem{MR2498845}
Yihong Du and Zongming Guo.
\newblock Positive solutions of an elliptic equation with negative exponent:
  stability and critical power.
\newblock {\em J. Differential Equations}, 246(6):2387--2414, 2009.

\bibitem{MR2604963}
Pierpaolo Esposito, Nassif Ghoussoub, and Yujin Guo.
\newblock {\em Mathematical analysis of partial differential equations modeling
  electrostatic {MEMS}}, volume~20 of {\em Courant Lecture Notes in
  Mathematics}.
\newblock Courant Institute of Mathematical Sciences, New York; American
  Mathematical Society, Providence, RI, 2010.

\bibitem{MR3867223}
Carlos Esteve and Philippe Souplet.
\newblock Quantitative touchdown localization for the {MEMS} problem with
  variable dielectric permittivity.
\newblock {\em Nonlinearity}, 31(11):4883--4934, 2018.

\bibitem{MR3945769}
Carlos Esteve and Philippe Souplet.
\newblock No touchdown at points of small permittivity and nontrivial touchdown
  sets for the {MEMS} problem.
\newblock {\em Adv. Differential Equations}, 24(7-8):465--500, 2019.

\bibitem{MR4375793}
Marius Ghergu and Yasuhito Miyamoto.
\newblock Radial single point rupture solutions for a general {MEMS} model.
\newblock {\em Calc. Var. Partial Differential Equations}, 61(2):Paper No. 47,
  29, 2022.

\bibitem{MR2286013}
Nassif Ghoussoub and Yujin Guo.
\newblock On the partial differential equations of electrostatic {MEMS}
  devices: stationary case.
\newblock {\em SIAM J. Math. Anal.}, 38(5):1423--1449, 2006/07.

\bibitem{MR2262209}
Hongxia Guo, Zongming Guo, and Ke~Li.
\newblock Positive solutions of a semilinear elliptic equation with singular
  nonlinearity.
\newblock {\em J. Math. Anal. Appl.}, 323(1):344--359, 2006.

\bibitem{MR3305368}
Jong-Shenq Guo and Philippe Souplet.
\newblock No touchdown at zero points of the permittivity profile for the
  {MEMS} problem.
\newblock {\em SIAM J. Math. Anal.}, 47(1):614--625, 2015.

\bibitem{MR2179754}
Yujin Guo, Zhenguo Pan, and M.~J. Ward.
\newblock Touchdown and pull-in voltage behavior of a {MEMS} device with
  varying dielectric properties.
\newblock {\em SIAM J. Appl. Math.}, 66(1):309--338, 2005.

\bibitem{MR3998250}
Yujin Guo, Yanyan Zhang, and Feng Zhou.
\newblock Singular behavior of an electrostatic-elastic membrane system with an
  external pressure.
\newblock {\em Nonlinear Anal.}, 190:111611, 29, 2020.

\bibitem{MR2351179}
Zongming Guo and Juncheng Wei.
\newblock On the {C}auchy problem for a reaction-diffusion equation with a
  singular nonlinearity.
\newblock {\em J. Differential Equations}, 240(2):279--323, 2007.

\bibitem{MR2393396}
Zongming Guo and Juncheng Wei.
\newblock Asymptotic behavior of touch-down solutions and global bifurcations
  for an elliptic problem with a singular nonlinearity.
\newblock {\em Commun. Pure Appl. Anal.}, 7(4):765--786, 2008.

\bibitem{MR2500854}
Zongming Guo and Juncheng Wei.
\newblock On solutions with point ruptures for a semilinear elliptic problem
  with singularity.
\newblock {\em Methods Appl. Anal.}, 15(3):377--390, 2008.

\bibitem{MR3265537}
Zongming Guo and Juncheng Wei.
\newblock Rupture solutions of an elliptic equation with a singular
  nonlinearity.
\newblock {\em Proc. Roy. Soc. Edinburgh Sect. A}, 144(5):905--924, 2014.

\bibitem{MR3426132}
ZongMing Guo and Feng Zhou.
\newblock Sub-harmonicity, monotonicity formula and finite {M}orse index
  solutions of an elliptic equation with negative exponent.
\newblock {\em Sci. China Math.}, 58(11):2301--2316, 2015.

\bibitem{MR4213008}
Yu~Ichida and Takashi~Okuda Sakamoto.
\newblock Radial symmetric stationary solutions for a {MEMS} type
  reaction-diffusion equation with spatially dependent nonlinearity.
\newblock {\em Jpn. J. Ind. Appl. Math.}, 38(1):297--322, 2021.

\bibitem{MR2326181}
Huiqiang Jiang and Wei-Ming Ni.
\newblock On steady states of van der {W}aals force driven thin film equations.
\newblock {\em European J. Appl. Math.}, 18(2):153--180, 2007.

\bibitem{MR2796243}
Meiyue Jiang, Liping Wang, and Juncheng Wei.
\newblock {$2\pi$}-periodic self-similar solutions for the anisotropic affine
  curve shortening problem.
\newblock {\em Calc. Var. Partial Differential Equations}, 41(3-4):535--565,
  2011.

\bibitem{kawarada1975solutions}
Hideo Kawarada.
\newblock On solutions of initial-boundary problem for ut= uxx+$\backslash$frac
  $\{$1$\}$$\{$1- u$\}$.
\newblock {\em Publications of the Research Institute for Mathematical
  Sciences}, 10(3):729--736, 1975.

\bibitem{MR3662914}
Philippe Lauren\c{c}ot and Christoph Walker.
\newblock Some singular equations modeling {MEMS}.
\newblock {\em Bull. Amer. Math. Soc. (N.S.)}, 54(3):437--479, 2017.

\bibitem{MR2161885}
Ke~Li, Hongxia Guo, and Zongming Guo.
\newblock Positive single rupture solutions to a semilinear elliptic equation.
\newblock {\em Appl. Math. Lett.}, 18(10):1177--1183, 2005.

\bibitem{MR1955412}
John~A. Pelesko and David~H. Bernstein.
\newblock {\em Modeling {MEMS} and {NEMS}}.
\newblock Chapman \& Hall/CRC, Boca Raton, FL, 2003.

\bibitem{MR0727393}
Renate Schaaf.
\newblock Global behaviour of solution branches for some {N}eumann problems
  depending on one or several parameters.
\newblock {\em J. Reine Angew. Math.}, 346:1--31, 1984.

\bibitem{MR4389601}
Qi~Wang and Yanyan Zhang.
\newblock Asymptotic and quenching behaviors of semilinear parabolic systems
  with singular nonlinearities.
\newblock {\em Commun. Pure Appl. Anal.}, 21(3):797--816, 2022.

\bibitem{MR2564407}
Dong Ye and Feng Zhou.
\newblock On a general family of nonautonomous elliptic and parabolic
  equations.
\newblock {\em Calc. Var. Partial Differential Equations}, 37(1-2):259--274,
  2010.

\end{thebibliography}

\end{document}